\documentclass[10pt]{amsart}
\usepackage{amsfonts}
\usepackage{ifthen}
\usepackage{amsthm}
\usepackage{amsmath}
\usepackage{graphicx}
\usepackage{amscd,amssymb,amsthm}
\usepackage{enumerate}
\usepackage{epstopdf}
\usepackage{mathrsfs}
\usepackage{hyperref}
\usepackage{float}
\usepackage[table,xcdraw]{xcolor}
\usepackage{multirow}
\usepackage{algorithm}
\usepackage{algorithmic}
\usepackage{tikz} 

\usepackage{subcaption}

\usepackage{enumitem}

\newcounter{minutes}
\divide\time by 60
\newcounter{hours}
\multiply\time by 60 \addtocounter{minutes}{-\time}

\usepackage{subcaption}

\title[]{Efficient third-order iterative algorithms for computing zeros of special functions} 

\author{Dhivya Prabhu K$^\dagger$}
\address{Discipline of Mathematics,
	Indian Institute of Technology Indore, Indore 453552, India}
\email{dhivyanslm@gmail.com}

\author{Sanjeev Singh$^\dagger$}
\address{Discipline of Mathematics,
	Indian Institute of Technology Indore, Indore 453552, India}
\email{snjvsngh@iiti.ac.in}

\author{Antony Vijesh V$^{\dagger,*}$}
\address{Discipline of Mathematics,
	Indian Institute of Technology Indore, Indore 453552, India}
\email{vijesh@iiti.ac.in}

\thanks{$^*$Corresponding author}
\thanks{E-mail addresses: dhivyanslm@gmail.com (Dhivya Prabhu K), snjvsngh@iiti.ac.in (Sanjeev Singh), vijesh@iiti.ac.in (Antony Vijesh V)}
\thanks{$^\dagger$ Department  of Mathematics,
	Indian Institute of Technology Indore, Indore 453552, India}

\newtheorem{theorem}{Theorem}
\newtheorem{remark}{Remark}
\newtheorem{lemma}{Lemma}

\begin{document}
	
	\def\thefootnote{}
	\footnotetext{ \texttt{File:~\jobname.tex,
			printed: \number\year-\number\month-\number\day,
			\thehours.\ifnum\theminutes<10{0}\fi\theminutes}
	} \makeatletter\def\thefootnote{\@arabic\c@footnote}\makeatother

	
	
	\maketitle
	\begin{abstract}	

       This manuscript presents a novel and reliable third-order iterative procedure for computing the zeros of solutions to second-order ordinary differential equations. By approximating the solution of the related Riccati differential equation using the trapezoidal rule, this study has derived the proposed third-order method. This work establishes sufficient conditions to ensure the theoretical non-local convergence of the proposed method. This study provides suitable initial guesses for the proposed third-order iterative procedure to compute all zeros in a given interval of the solutions to second-order ordinary differential equations. The orthogonal polynomials, like Legendre and Hermite, as well as the special functions, like Bessel and cylinder functions, satisfy the proposed conditions for convergence. Numerical simulations demonstrate the effectiveness of the proposed theory. This work also presents a comparative analysis with recent studies. 	
	\end{abstract}
	
	\section{\bf Introduction}
	Developing algorithms to compute the zeros of nonlinear functions is one of the demanding problems in various disciplines. Finding the zeros of special functions and orthogonal polynomials is especially common in many applications. For example, the zeros of orthogonal polynomials play a crucial role in the spectral method for solving differential equations and the quadrature rule for evaluating integrals (for example, see \cite{app_op-1, OST20}). Similarly, the zeros of the Bessel function regularly appear in scattering theory and quantum mechanics \cite{app_bf-1, app_bf-2}. From the literature, it is evident that consistent efforts were made to find all the zeros of Legendre polynomial \cite{i3,  ni2, alex2, i4, t5, i5}, Hermite polynomial \cite{seg1, i1, Comp.phy}, Bessel function \cite{matrix, bremer, Bessel24, i1}, and cylinder function \cite{ 1998, seg2, 2010} using various techniques. It is worth mentioning that in all these situations, one typically has to find more than one zero. In many places, special functions have an infinite number of zeros. Thus, finding all the zeros in a given interval is a crucial problem that attracts many researchers to investigate it.

	Although several techniques are available in the literature, approximating zeros using an iterative procedure is a popular approach among researchers. As all iterative procedures begin with an initial guess, finding a suitable initial guess is a crucial task. This task becomes quite challenging when attempting to find multiple zeros of a nonlinear function. In such a situation, either one has to determine a suitable subdomain containing precisely one of the zeros such that the iterative method has a convergence property in that subdomain or an initial guess sufficiently close to the corresponding zero such that the iterative method converges to that zero. We call this step bracketing. Hence, developing an iterative algorithm with a built-in bracketing property to find all the zeros of a function is a vital problem. Moreover, as the behaviour of the location of zero changes immensely for different functions, developing unified iterative algorithms with a built-in bracketing property that can handle a large class of functions is an interesting and challenging problem.
	
	Newton's method is one of the iterative methods widely used for finding the zeros of nonlinear functions. Newton's method has a non-local, second-order convergence property if the nonlinear function is convex and has a unique zero. In some situations, Newton's method is hypersensitive to the initial guess (for example, see \cite{book}). To overcome this issue, most modern algorithms \cite{i3, i1, alex2, alex1} employ different techniques to obtain an initial guess for each zero. Later, Newton's method is used to get the refined version of the zeros. In \cite{i3}, an algorithm was studied for finding Gauss-Legendre quadrature nodes. In this algorithm, to obtain the $i$-th node corresponding to the $n$-th degree polynomial, Newton's method uses $\pi\frac{4i+3}{4n+2}$  as the initial guess, which is obtained from the asymptotic formula \cite{asf1}. In \cite{i1}, an algorithm was developed to find all the zeros of the Bessel function, Hermite polynomial, Legendre polynomial, and prolate spheroidal wave function. First, using the Pr$\ddot{\mathrm{u}}$fer transform, an initial value problem is obtained corresponding to each function/polynomial. Assuming the first zero, a crude approximation for the rest of the zeros is obtained successively by solving the initial value problem using either the Runge-Kutta or the Taylor series method. This crude approximation was used as an initial guess for Newton's method. In \cite{alex2}, to find Gauss-Legendre and Gauss-Jacobi quadrature nodes, Newton's method is used with an initial guess supplied from asymptotic formulas involving zeros of the Bessel function. In \cite{alex1}, an algorithm was studied for finding Gauss-Hermite quadrature nodes using the Newton method. In this algorithm, to obtain an initial guess, one must either solve a nonlinear equation or use an asymptotic formula based on the zeros of the Airy function.

	An interesting Newton method was studied in \cite{1998}, to find all the zeros of the cylinder function $C_\nu(x)$, $\nu \in \mathbb{R}$. This method is based on the fact that the zeros of $C_\nu(x)$ and $C_{\nu-1}(x)$ are interlaced. Consequently, the function $H_\nu(x)=\frac{C_\nu(x)}{C_{\nu-1}(x)}$ has same zeros as $C_\nu(x)$ and between two successive zeros of $C_{\nu-1}(x)$, $H_\nu(x)$ has exactly one zero. Using the relation between $C_\nu(x)$ and $C_{\nu-1}(x)$, it was shown that $H_\nu(x)$ satisfies an Riccati equation. Later, it was shown that between two successive zeros of $C_{\nu-1}(x)$, the Newton method has a nonlocal convergence property for the function $f_\nu(x)=x^{2\nu-1}H_\nu(x)$. In other words, between the singular points of $f_\nu(x)$, the second-order method has global convergence. Interestingly, the zeros of many special functions and orthogonal polynomials have this interlacing property. Thus, the idea in \cite{1998} was extended to special functions and orthogonal polynomials, which are solutions of second-order differential equations in \cite{2002}. More specifically,  in \cite{2002}, to find the zero of $y_n(x)$, the ratio $\frac{y_n(x)}{y_{n\pm1}(x)}$ was used. Similar to \cite{1998}, the function $Y(x)=\frac{y_n(x)}{y_{n\pm1}(x)}$ satisfy a Riccati differential equation. A new second-order iterative procedure for finding the zero of $\frac{y_n(x)}{y_{n\pm1}(x)}$ involving the $\arctan$ function is obtained using the Riccati equation. Also, it was shown that the new second-order method has a nonlocal convergence property between two successive zeros of $y_{n\pm1}$. i.e. between the singular points of $Y(x)$, the second-order method has global convergence. An important result related to the distance between the adjacent zeros of $y_n(x)$ and $y_{n\pm1}(x)$ is also provided. Using this result by adding or subtracting $\frac{\pi}{2}$ with the current zero of $y_n(x)$, one can get the initial values for the next zero of $y_n(x)$ for the new iterative method. Hence, the second-order iterative algorithm in \cite{2002} has a built-in bracketing property. This second-order iterative method is further extended in \cite{2003} to an arbitrary pair of the function $(y(x), w(x))$ such that the zeros of $y(x)$ and $w(x)$ are interlaced and satisfy a suitable first-order coupled differential equation(CDE) of the form
	\begin{equation}\label{02}
		\left\{\begin{aligned}
			y'(x)&=c_1(x)y(x)+c_2(x)w(x)\\
			w'(x)&=c_3(x)w(x)+c_4(x)y(x),
		\end{aligned}\right.
	\end{equation}
	where $c_1(x)$, $c_2(x)$, $c_3(x)$ and $c_4(x)$ are continuous functions. More specifically,  in \cite{2003}, to find the zero of $y(x)$, the function $\mathcal{Y}(x)=\frac{y(x)}{w(x)}$ was used, and the rest of the discussion is similar to \cite{2002}. Also, it was shown that the second-order method has a nonlocal convergence property between two successive zeros of $w(x)$. i.e. between the singular points of $\mathcal{Y}(x)$, the second-order method has global convergence. It is worth mentioning that in \cite{2003}, an improved version of the distance between the adjacent zeros of $y(x)$ and $w(x)$ is also provided. The second-order iterative algorithm discussed in \cite{2003} has a built-in bracketing property. A variety of applications are demonstrated in \cite{2003,2002}.

    A higher-order iterative algorithm was developed in \cite{2010} with a built-in bracketing property, following a similar direction to \cite{2003, 2002}. More specifically, to find the zero of $y(x)$, the ratio $\mathbb{Y}(x)=\frac {y(x)}{y'(x)}$ was considered, and $y(x)$ was assumed to be a non-trivial solution of a second-order linear differential equation $y''(x)+A(x)y(x)=0$. Similar to \cite{2003,1998,2002}, the ratio $\frac{y(x)}{y'(x)}$ satisfy a Riccati differential equation. By evaluating the integral arising from the corresponding Riccati equation semi-analytically, a fourth-order iterative procedure for computing all the zeros of the function $\frac{y(x)}{y'(x)}$ is derived in \cite{2010}. In contrast to the methods in \cite{2003,2002}, an alteration is necessitated in the iterative procedure \cite{2010} to ensure convergence from the initial guess, which is based on the bounds on the distances between adjacent zeros of $y(x)$ and $y'(x)$. The altered iterative procedure \cite{2010} retains its fourth-order accuracy and possesses a built-in bracketing property. Other interesting iterative approaches for finding multiple roots of polynomials and functions can be found in \cite{BPBC25,CN26}.

    This manuscript aims to address the following problem: ``Under the same conditions on the singular points as in \cite{2003, 2002}, does there exist a higher order iterative method having properties similar to \cite{2003, 2002}"? This study answers this question favourably by constructing a third-order iterative method with properties similar to those in \cite{2003, 2002}. More specifically, a novel, easily implementable third-order iterative method with a built-in bracketing property is proposed to find all the zeros of various types of special functions and orthogonal polynomials within a given interval. The new third-order iterative method is derived using the same assumptions of \cite{2003}. i.e. $y(x)$ and $w(x)$ are arbitrary functions such that the zeros of $y(x)$ and $w(x)$ are interlaced and satisfying \eqref{02}. To find the zero of $y(x)$, the ratio $\frac{y(x)}{w(x)}$ is used and reaches the Riccati equation as in \cite{2003}. The proposed novel third-order iterative method is obtained by approximating the solution of the Riccati equation by the trapezoidal rule, a key difference from the approaches in \cite{2003, seg1, seg11, log, 2002, 2010} and similar to that in \cite{CKL18, WF00}. Interestingly, the proposed novel third-order iterative method possesses a built-in bracketing property, as in \cite{2003, 2002, 2010}. Theoretically, it is shown that the proposed third-order iterative method has a global convergence property between two successive singular points. 
    The main contributions of the paper can be highlighted as follows:
	\begin{itemize}
		\item An easily implementable novel third-order iterative method is proposed for finding all the zeros of the solution $y(x)$ of a second-order linear ODE.
		\item The global convergence of the proposed iterative method has been established under suitable assumptions.
		\item The proposed iterative method has a built-in bracketing property.
        \item Theoretically, it was shown that the well-known Legendre and Hermite polynomials and the Bessel, cylinder functions satisfy the proposed conditions for convergence.
        \item A detailed numerical simulation results were presented by handling Legendre polynomials, Hermite polynomials, Bessel functions and cylinder functions. 
        \item A comparative study with the recent methods in the literature \cite{alex1} and \cite{seg1,seg11} is also presented. 
	\end{itemize}
	
	The remainder of this paper is organised as follows: Section 2 presents some required preliminaries from the earlier literature. The construction of the proposed novel third-order iterative procedure is also given in Section 2. Section 3 displays the global convergence theorems for the proposed iterative procedure in finding all zeros within an interval. Section 4 displays the application of the proposed iterative method to find zeros of  Legendre and Hermite polynomials, as well as the Bessel and cylinder functions. Section 5 presents the numerical comparative study for the problem of finding all zeros of Hermite polynomials, Legendre polynomials, and finding all zeros of Bessel and cylinder functions in a given interval. We conclude our discussion in Section 6 by highlighting the summary of the work.

	\section{Preliminaries}
	This section presents the construction of a new iteration method \eqref{iter} for finding all zeros of the function $y(x)$ in the interval $I$, provided that there exist another function $w(x)$ such that the pair  $(y(x), w(x))$ satisfies the system of first order coupled differential equation \eqref{02} in the interval $I$, with the following conditions
	\begin{enumerate}
		\item Zeros of $y(x)$ and $w(x)$ are simple.
		\item $c_2(x)\neq0$ and $c_4(x)\neq 0$, $\hspace{3mm}\forall \hspace{2mm}x \in I.$
		\item Zeros of $y(x)$ and $w(x)$ are interlaced and $c_2(x)c_4(x)<0$, $\hspace{3mm}\forall \hspace{2mm}x \in I.$
	\end{enumerate}
	Conditions (1)-(3) are essential for constructing the iterative scheme. It is interesting to note that the existence of the function $w(x)$, with the three conditions mentioned above, are immediate consequence once we have the assumption that CDE \eqref{02} is satisfied simultaneously by the pairs $(y_1(x), w_1(x))$ and $(y_2(x), w_2(x))$, where $(y_1(x), y_2(x))$ and $(w_1(x), w_2(x))$ are fundamental solutions of ODEs \eqref{ode1} and \eqref{ode2}, respectively and $y(x)$ is a non-trivial solution of \eqref{ode1}. The following theorem explicitly presents this fact. 
	\begin{theorem}\cite{2003, 2002}\label{thm1}
		Let $A_{y}(x),$ $B_{y}(x),$ $A_w(x),$ $B_w(x),$ $c_1(x),$ $c_2(x),$ $c_3(x),$ and $c_4(x)$ are real valued continuous functions on $I.$ Let $\{y_1(x), y_2(x)\}$ and $\{w_1(x), w_2(x)\}$, $x \in I$ be fundamental sets of solutions of the second-order linear ODEs
		\begin{equation}\label{ode1}
			f''(x)+B_{y}(x)f'(x)+A_{y}(x)f(x)=0
		\end{equation} and
		\begin{equation}\label{ode2}
			f''(x)+B_{w}(x)f'(x)+A_{w}(x)f(x)=0,
		\end{equation}
		respectively. Further assume that $(y_1(x), w_1(x))$ and $(y_2(x), w_2(x))$ are solutions of coupled differential equation (CDE) \eqref{02} on $I$. If $y(x)$ is a non-trivial solution of \eqref{ode1}. Then, there exists a non-trivial function $w(x)$ which is a solution of \eqref{ode2} and the pair $(y(x), w(x))$ satisfies the CDE \eqref{02}. Furthermore, the following statements also hold true:
		\begin{enumerate}
			\item Zeros of $y(x)$ and $w(x)$ are simple.
			\item $c_2(x) \neq 0$ and $c_4(x) \neq 0$,  $\hspace{3mm}\forall \hspace{2mm}x \in I$.
			\item Zeros of $y(x)$ and $w(x)$ are interlaced, and $c_2(x)c_4(x)<0$,  $\hspace{2mm}\forall \hspace{2mm}x \in I$ provided one of the functions, either $y(x)$ or $w(x)$, has at least two zeros in $I$.
		\end{enumerate}
	\end{theorem}

	\textbf{Assumption:} Throughout this paper we will assume that the pair $(y(x), w(x))$ satisfies the CDE \eqref{02} in the interval $I$ and the statements \textit{(1), (2)} and \textit{(3)} in Theorem \ref{thm1} also hold true.
	\subsection{Construction of third-order iterative method}\label{2}
	First, we derive a Riccati differential equation using $y(x)$ and $w(x)$ in the CDE \eqref{02}. To proceed further, we introduce the following change of variable mentioned in \cite{2003, 2002}. Let $y(x)=\tilde{y}(x)$ and $w(x)=\text{sign}(c_2(x))k(x)\tilde{w}(x),$ where $k(x)=\sqrt{-\frac{c_4(x)}{c_2(x)}}$. Consequently, the zeros of $\tilde{y}(x)$ and $\tilde{w}(x)$ are the same as the zeros of $y(x)$ and $w(x)$, respectively. Note that $(\tilde{y}(x),\tilde{w}(x))$ satisfy the following CDE 
	\begin{equation}\label{dde2}
		\left\{\begin{aligned}
			\tilde{y}'(x)&=\tilde{c}_1(x)\tilde{y}(x)+\tilde{c}_2(x)\tilde{w}(x)\\
			\tilde{w}'(x)&=\tilde{c}_3(x)\tilde{w}(x)+\tilde{c}_4(x)\tilde{y}(x),
		\end{aligned}\right.
	\end{equation}
	where $\tilde{c}_1(x)=c_1(x),\hspace{3mm} \tilde{c}_2(x)=|c_2(x)|k(x),\hspace{3mm} \tilde{c}_3(x)=\left[c_3(x)-\dfrac{k'(x)}{k(x)}\right]$, and $\hspace{2mm}\tilde{c}_4(x)=\text{sign}(c_2)\dfrac{c_4(x)}{k(x)}$. Hence, $\tilde{c}_2(x)>0$ and $\dfrac{\tilde{c}_2(x)}{\tilde{c}_4(x)}=-1$.
	
	Let $w_*$ and $w_{**}$ be two successive zeros of $w(x)$. Define $t(x)=\dfrac{\tilde{y}(x)}{\tilde{w}(x)}$ $\hspace{2mm}\forall\hspace{2mm}x \in J=(w_*, w_{**})$. Clearly, $t(x)$ is well defined on $J$ and $t(x)$ satisfies the following Riccati equation on $J$ 
	\begin{equation}\label{r1}
		t'(x)=\tilde{c}_2(x)(1+t^2(x))+(\tilde{c}_1(x)-\tilde{c}_3(x))t(x).
	\end{equation}
	Introduce the change of variable $z(x)=\displaystyle\int\tilde{c}_2(x) \mathrm{d}x$ as in \cite{2003, 2002}. Let $x(z)$ be the inverse of $z(x).$ Define $h(z)=t(x(z))$. Hence, $z^*$ is a zero of $h(z)$ if and only if $x(z^*)$ is a zero of $t(x)$. Thus, $z^*$ is a zero of $h(z)$ if and only if $x(z^*)$ is a zero of $y(x)$. Moreover, inverse mapping theorem ensures that $\dfrac{{\rm d}x}{{\rm d}z}=\dfrac{1}{\tilde{c}_2(x)}$. Denote $\dfrac{{\rm d}h}{{\rm d}z}$ as $\dot{h}(z)$. Hence, \eqref{r1} becomes
	\begin{equation}\label{ricatti}
		\dot{h}(z)=1+h^2(z)-2r(z)h(z),
	\end{equation} where $r(z)=\dfrac{\tilde{c}_3(x(z))-\tilde{c}_1(x(z))}{2\tilde{c}_2(x(z))}.$ Let $z^*$ be a zero of $h(z)$. Integrating the above equation \eqref{ricatti} from $z^*$ to $z$, we get
	\begin{equation*}
		h(z)=\int_{z^*}^{z}(1+h^2(t)-2r(t)h(t)) {\rm d}t.
	\end{equation*}
	Approximating the above integral by using the trapezoidal rule, we obtain
	\begin{align*}
		h(z)&\approx \frac{1}{2}(z-z^*)(2+h^2(z)-2r(z)h(z))\\
		z^*&\approx z-\dfrac{2h(z)}{2+h^2(z)-2r(z)h(z)}.
	\end{align*}
	Thus, if $z$ is an approximation for $z^*$, one can expect 
	\begin{equation}\label{g}
		G(z)=z-\dfrac{2h(z)}{2+h^2(z)-2r(z)h(z)}
	\end{equation}
	maybe a better approximation of $z^*$. Hence, we obtain a new iterative scheme $z_{n+1}=G(z_n)$, i.e.,
	\begin{equation} \label{iter}
		z_{n+1}=z_n-\dfrac{2h(z_n)}{2+h^2(z_n)-2r(z_n)h(z_n)}, \hspace{3mm} n=0,1,2,\dots
	\end{equation}
	for approximating the zero of $h(z)$.
	\begin{remark}
	\end{remark}
	
	\textbf{$\bullet$} It is easy to verify that $G(z^*)=z^*$, $\dot{G}(z^{*})=0$, $\ddot{G}(z^{*})=0$. Define $\epsilon_n=z_n-\alpha.$ Consequently,  $\epsilon_{n+1}=\dfrac{\dddot{G}(z^*)}{6}\epsilon_n^3+\mathcal{O}(\epsilon_n^4)$. Hence, the proposed iterative procedure \eqref{iter} is a third-order method.
	
	\textbf{$\bullet$} From \eqref{ricatti}, the Newton's method for finding the zeros of $h(z)$ can be expressed as
	\begin{equation}\label{iter1}
		z_{n+1}=z_n-\dfrac{h(z_n)}{1+h^2(z_n)-2r(z_n)h(z_n)}, \hspace{3mm} n=0,1,2,\dots.
	\end{equation}
	It is interesting to note that the proposed algorithm bears a resemblance to the well-known Newton's method; the proposed method exhibits a higher order of convergence than Newton's method. 

    The bound for the distance between the singular point and the zero plays a crucial role in providing initial guesses for zeros of $h(z)$. Using the Riccati Equation \ref{ricatti}, one can get these bounds \cite{2003}. Let $z_w$ and $z'_w$ be two consecutive singularities of $h(z)$. Equation (2.15) of \cite{2003} can be formally stated as
	\begin{lemma} \label{lemma2}
		Let $r(z) \neq 0$ in the interval $(z_w, z'_w)$ and $z_*$ be the unique zero of $h(z)$ in $(z_w, z'_w)$.The following statements are then true.
		\begin{enumerate}
			\item If $r(z)>0$ for all $z \in (z_w,z'_w)$, then $z_{*}-z_{w}<\frac{\pi}{2}$ and $z_w'-z_{*}>\frac{\pi}{2}$.
			\item If $r(z)<0$ for all $z \in (z_w,z'_w)$, then $z_{*}-z_{w}>\frac{\pi}{2}$ and $z_w'-z_{*}<\frac{\pi}{2}$.
		\end{enumerate}
	\end{lemma}
	The following lemma provides a similar bound when $r(z)$ has a unique zero $z_r$ in $(z_w,z'_w)$.
	\begin{lemma}\cite[Proposition 6.1]{2002}\label{lemma3}
		Let $r(z)$ be a non-increasing and has a unique zero $z_r$ in $(z_w, z'_w)$ and $z_*$ be the unique zero of $h(z)$ in $(z_w, z'_w)$. The following statements are then true.
		\begin{enumerate}
			\item If $z_r> z_*$, then $z_w'-z_r<\frac{\pi}{2}$.
			\item If $z_r<z_*$, then $z_r-z_w<\frac{\pi}{2}$.
		\end{enumerate}
	\end{lemma}
The distance between consecutive zeros also plays a vital role in providing initial guesses. If the zeros of $y(x)$ have the property of being either convex or concave \cite{convex}, one can obtain improved initial guesses. The following well-known "Sturm convexity theorem" is used in \cite{2003} for obtaining improved initial guesses.
\begin{theorem}\cite[p. 318]{BR89}\label{sturm_lemma}
Consider the differential equation $u''(x)+\Omega(x)u(x)=0$ where $\Omega(x)$ is an increasing continuous function. Then $x_n-x_{n-1} < x_{n+1}-x_n$, where $x_n$ is the sequence of successive zeros of a nontrivial solution of the differential equation.
\end{theorem}
It is worth mentioning that if $\Omega(x)$ is decreasing, then the inequality in the above theorem will be reversed. We will conclude this section by presenting the following remark.
\begin{remark}\label{remark_prelim_2}
    Using a suitable nonzero function $d(z)$, it was shown in \cite[Eq. 2.11]{2003} that the function $\overline{y}(z)=d(z)\tilde{y}(z)$ satisfies the following second-order differential equation $\ddot{\overline{y}}(z)+\Omega(z)\overline{y}=0,$ where $\Omega(z)=1+\dot{r}(z)-r^2(z)$. Note that the zeros of $y(z)$, $\tilde{y}(z)$ and $\overline{y}(z)$ are same.
\end{remark}

	\section{Convergence analysis of the proposed scheme}
	This section discusses the global convergence behaviour of the proposed third-order iterative scheme. Throughout this section, assume that $z_w<z_w'$ are two successive zeros of $\tilde{w}(x(z))$. Let $z_{*}$ denotes the unique zero of $\tilde{y}(x(z))$ in $J'=(z_w, z_w')$. We assume that $r(z)$ is a differentiable function on $J'$. Based on the behaviour of the function $r(z)$, the global convergence property of the proposed iterative method \eqref{iter} has been established in the interval $J'$. More specifically, under suitable assumptions on $r(z)$, the sequence $(z_n)$, $n=0,1,2,\dots$ generated by the proposed scheme \eqref{iter} converges monotonically to $z_*$ in $J'$. Using the convergence results, efficient algorithms have been proposed to find all the zeros of $\tilde{y}(x(z))$ in a given interval $[a, b] \subseteq (z_{w_j}, z_{w_k})$, where $z_{w_j}<z_{w_k}$ are zeros of $\tilde{w}(x(z))$. The following remark provides an interesting property of the function $h(z)$ in $J'$, which plays a crucial role in the convergence analysis.
	\begin{remark}\label{signh}
		The property $\dot{h}(z_*)=1,$ guarantee that $h(z)<0\hspace{2mm}\forall\hspace{2mm}z \in (z_{w}, z_*)$ and $h(z)>0\hspace{2mm}\forall\hspace{2mm}z \in (z_*, z_w')$. 
	\end{remark} 
	The following theorem provides sufficient conditions for the monotone convergence of the proposed scheme in the interval $J'$.
	\begin{theorem}\label{thm2}
		Let $r(z)$ be a non-increasing function on $J'$.
		\begin{enumerate}
			\item If $r(z)$ is positive in $(z_w, z_*)$, then the sequence $(z_n)$ generated by the proposed iterative method \eqref{iter} converges monotonically to $z_*$ for any initial guess $z_0 \in (z_w, z_*)$. 
			\item If $r(z)$ is negative in $(z_*, z_w')$, then the sequence $(z_n)$ generated by the proposed iterative method \eqref{iter} converges monotonically to $z_*$ for any initial guess $z_0 \in (z_*, z_w')$.
		\end{enumerate}
	\end{theorem}	
	
	\begin{proof}
		From Remark \ref{signh}, it is clear that $r(z)h(z)<0 \hspace{3mm}\forall\hspace{2mm}z \in (z_w, z_{*})$. Hence the function $2+h^2(z)-2r(z)h(z)>0 \hspace{3mm}\forall\hspace{2mm} z \in (z_w, z_{*})$. Consequently, the scheme (\ref{iter}) is well-defined $\forall z \in (z_w, z_{*})$. Consider $z_0 \in (z_w, z_{*})$, $$z_1=G(z_0)=z_0-\dfrac{2h(z_0)}{2+h^2(z_0)-2r(z_0)h(z_0)}$$ which is well-defined, and it is easy to verify that $z_0<z_1$. By using \eqref{g}, we observe that
		\begin{equation}\label{g'}
			\dot{G}(z)=\dfrac{3h^4(z)+2h^2(z)+4r^2(z)h^2(z)-4\dot{r}(z)h^2(z)-8r(z)h^3(z)}{(2+h^2(z)-2r(z)h(z))^2}
		\end{equation} is non-negative in $(z_{w}, z_{*}).$
		Thus,
		\begin{equation*}
			z_{*}-z_1=G(z_{*})-G(z_0)=\dot{G}(c)(z_{*}-z_0)>0, \hspace{2mm}\text{for some}\hspace{2mm}c \in (z_0, z_{*}).
		\end{equation*}
		Consequently, $z_1<z_{*}$. Consider this conclusion is valid for $k=1,2,\dots,n$. i.e., $z_1,z_2,\dots,z_n$ exist and $z_{w}<z_0<z_1<\dots<z_n<z_{*}$. Now,  $z_{n+1}=G(z_n)=z_n-\dfrac{2h(z_n)}{2+h^2(z_n)-2r(z_n)h(z_n)}$ is well-defined and it is easy to verify that $z_n<z_{n+1}$.
		\begin{equation*}
			z_{*}-z_{n+1}=G(z_{*})-G(z_n)=\dot{G}(c)(z_{*}-z_n)>0, \hspace{2mm}\text{for some}\hspace{2mm}c \in (z_n,z_{*}).
		\end{equation*}
		Consequently, $z_{n+1}<z_{*}$. Thus, the sequence $(z_n)$ generated by the scheme (\ref{iter}) is monotone increasing and bounded above by $z_{*}$. By using \eqref{iter}, it is easy to verify that $(z_n)$ converges to the zero $z_{*}$ of $h(z)$ in $J'$. The proof for the second part of the theorem follows similarly.
	\end{proof}
	Note that the above theorem does not guarantee the convergence of the sequence if $r(z)\geq 0$ and $z_0 \in (z_*, z_w')$. Similarly, there is no information regarding the convergence if $r(z)\leq 0$ and $z_0 \in (z_w, z_*)$. We can answer this question favourably with additional assumptions on $r(z)$. The following lemma plays a crucial role in this direction.
	
	\begin{lemma}\label{l1}
		Let $r(z)$ be non-increasing in the interval $J'$.
		\begin{enumerate}
			\item If $0<r(z)<1$ in $(z_*, z_w')$, then $\dot{G}(z)\geq0$ $\hspace{2mm}\forall\hspace{2mm}z \in (z_*, z_w')$.
			\item If $-1<r(z)<0$ in $(z_w, z_*)$, then $\dot{G}(z)\geq0$ $\hspace{2mm}\forall\hspace{2mm}z \in (z_w, z_*)$.
		\end{enumerate}
	\end{lemma}
	\begin{proof}
		\textit{(1).}	For $k \in \mathbb{N}$, define $S_k=(z_*, z_w')\cap r^{-1}\left(\left(\dfrac{2^{k-1}-1}{2^{k-1}}, \dfrac{2^k-1}{2^k}\right]\right)$. Note that $\cup_{k\in \mathbb{N}}S_k=(z_*, z_w')$. Let $z \in (z_*, z_w')$. Then, there exist $k \in \mathbb{N}$, such that $\dfrac{2^{k-1}-1}{2^{k-1}}<r(z)\leq\dfrac{2^k-1}{2^k}.$ Note that \text{sign}($\dot{G}(z)$)=\text{sign}($g(z)$), where $g(z)=3h^4(z)+2h^2(z)+4r^2(z)h^2(z)-4\dot{r}(z)h^2(z)-8r(z)h^3(z)$. Now, note that 
		\begin{align*}
			g(z)&\geq 3h^4(z)+2h^2(z)+4r^2(z)h^2(z)-8r(z)h^3(z)\geq h^2(z)\mathrm{f}_k(h),
		\end{align*}
		where $\mathrm{f}_k(h)=3h^2-8\left(\dfrac{2^k-1}{2^k}\right)h+\left(2+4\left(\dfrac{2^{k-1}-1}{2^{k-1}}\right)^2\right).$ Now, we will evaluate the discriminant $D$ of $\mathrm{f}_k(h)$. Note that
		\begin{align*}
			D&=64\left(\dfrac{2^k-1}{2^k}\right)^2-12\left(2+4\left(\dfrac{2^{k-1}-1}{2^{k-1}}\right)^2\right)=-\dfrac{8}{2^{2k}}\left(2^{2k}-8\cdot2^{k}+16\right)=-\dfrac{8}{2^{2k}}(2^k-4)^2\leq0.
		\end{align*}
		Thus, for $k \neq 2$, $\mathrm{f}_k(h)>0$ $\hspace{2mm}\forall\hspace{2mm}z \in (z_*, z_w')$. Consequently, for $k \neq 2$, $\dot{G}(z)\geq0 \hspace{2mm}\forall\hspace{2mm}z \in(z_*, z_w')$. For the case $k=2$, we have $\mathrm{f}_2(h)=3h^2-6h+3=3(h-1)^2\geq 0$. Hence, in this case also $\dot{G}(z)\geq0.$ Thus, $\hspace{2mm}\forall\hspace{2mm} z \in (z_*, z_w'),$ $\dot{G}(z)\geq0$.
		
		\textit{(2).} The proof is similar to the previous part. Express $(z_w, z_*)=\cup_{k\in \mathbb{N}}S'_k$, where $S'_k=(z_w, z_*)\cap r^{-1}\left(\left[-\dfrac{2^k-1}{2^k}, -\dfrac{2^{k-1}-1}{2^{k-1}}\right)\right)$ and proceed like part \textit{(1)} one can get the conclusion.
	\end{proof}
	
	Using Lemma \ref{l1}, now we will prove the convergence of the iterative scheme for the cases $r(z)$ is positive in $(z_*, z_w')$ and $r(z)$ is negative in $(z_w, z_*)$.
	\begin{theorem}\label{thm5}
		Let $r(z)$ be a non-increasing function in $J'$.
		\begin{enumerate}
			\item If $0<r(z)<1$ in $(z_*, z_w')$, then the sequence $(z_n)$ generated by the proposed iterative method \eqref{iter} converges monotonically to $z_*$ for any initial guess $z_0 \in (z_*, z_w')$.
			\item If $-1<r(z)<0$ in $(z_w, z_*)$, then the sequence $(z_n)$ generated by the proposed iterative method \eqref{iter} converges monotonically to $z_*$ for any initial guess $z_0 \in (z_w, z_*)$.
		\end{enumerate}
	\end{theorem}
	\begin{proof}
		First, we will prove Part \textit{(1)}. Observe that $h(z)>0$ and $-2r(z)h(z)>-2h(z)$ for all $z \in (z_{*}, z_w')$. Hence,
		\begin{equation*}
			2+h^2(z)-2r(z)h(z)>1+(1-h(z))^2>0, \quad \forall z \in (z_{*}, z_w').
		\end{equation*} 
		Consider $z_0 \in (z_{*}, z_w')$, then $z_1=G(z_0)=z_0-\dfrac{2h(z_0)}{2+h^2(z_0)-2r(z_0)h(z_0)}$ is well-defined and $z_1<z_0$. Note that
		\begin{equation*}
			z_{*}-z_1=G(z_{*})-G(z_0)=\dot{G}(c)(z_{*}-z_0) \quad \text{for some} \quad c\in (z_{*}, z_0).
		\end{equation*}	
		Using Lemma \ref{l1}, one can conclude that $z_*-z_1<0$. Hence, $z_*<z_1<z_0$. Consider this conclusion is valid for $k=1,2,\dots,n$. i.e., $z_1,z_2,\dots,z_n$ exist and satisfy $z_{*}<z_n<z_{n-1}<\dots<z_0<z_w'$. Clearly, $z_{n+1}=G(z_n)$ is well-defined and $z_{n+1}<z_n$. Note that
		\begin{equation*}
			z_{*}-z_{n+1}=G(z_{*})-G(z_n)=\dot{G}(c)(z_{*}-z_n) \quad \text{for some} \quad c \in (z_{*}, z_n).
		\end{equation*}
		Consequently, $z_{*}<z_{n+1}$. Hence, the sequence $(z_n)$ obtained from the scheme (\ref{iter}) is monotonically decreasing and bounded below by $z_{*}$. Hence, $(z_n)$ will converge to a fixed point $z_{*}$ of $G(z)$.
		
		For Part \textit{(2)}, one can check that
		\begin{equation*}
			2+h^2(z)-2r(z)h(z)>2+h^2(z)+2h(z)=1+(1+h(z))^2>0 \quad \forall z \in (z_w, z_{*}).
		\end{equation*}
		Hence, the iterative method is well-defined in the interval $(z_w, z_{*})$. The proof for part \textit{(2)} is analogous to part \textit{(1)}. Hence, it is left out.
	\end{proof}	
	By clubbing the conclusions from Theorem \ref{thm2} and Theorem \ref{thm5}, we can get the global convergence property of the proposed iterative method in the interval $J'$, when $\dot{r}(z)\leq 0$. This observation is stated as a remark.

	\begin{remark}\label{remark3}
		Let $r(z)\neq0$ and non-increasing in $J'$. If  $\left|r(z)\right|<1$, then the sequence $(z_n)$ generated by the proposed iterative method \ref{iter} converges monotonically to $z_*$ for any initial guess $z_0 \in J'$.
	\end{remark}
	The hypotheses of Theorem \ref{thm2} and Theorem \ref{thm5} are carefully assuming that the function $r(z)$ does not have a zero in $J'$. The following theorem addresses this issue.
	\begin{theorem}\label{theorem4}
		Let $r(z)$ be non-increasing and $\left|r(z)\right|<1$ in $J'$. Let $z_r$ be a unique zero of $r(z)$ in $J'$. Then, the sequence $(z_n)$ generated by the proposed iterative method \ref{iter} converges monotonically to $z_*$ for any initial guess $z_0$ between $z_r$ and $z_*$. More specifically, in this case, $z_r$ is a suitable initial guess for the proposed iterative procedure.
	\end{theorem}
	\begin{proof}
		Given $z_r \in J'$. Consequently, either $z_r<z_*$ or $z_*<z_r$. Consider the case $z_r<z_*$. Note that $r(z)\leq 0$, $h(z)\leq0$ and $\left|r(z)\right|<1$ in $[z_r, z_*)$. Hence, the function $G(z)$ is well-defined in $[z_r, z_*)$. Similarly to part (2) of Lemma \ref{l1}, one can conclude that $\dot{G}(z)\geq 0$ in $[z_r, z_*)$. Let $z_0 \in [z_r, z_*)$. Then, $z_1$ exists. From the definition of $z_1$, one can get $z_0<z_1$. Further, 
		$$z_1-z_*=G(z_0)-G(z_*)=\dot{G}(c)(z_0-z_*)\hspace{3mm}\text{for some}\hspace{2mm}c \in (z_0, z_*).$$ Hence, $z_0<z_1<z_*$. Consider this conclusion is valid for $k=1,2,\dots,n.$ i.e., $z_1, z_2, \dots, z_n$ exist and satisfy $z_1<z_2<\dots<z_n<z_*$. Clearly, $z_{n+1}=G(z_n)$ exist and $z_n<z_{n+1}$. Note that $$z_{n+1}-z_*=G(z_n)-G(z_*)=\dot{G}(c)(z_n-z_*)\hspace{3mm}\text{for some}\hspace{2mm}c \in (z_n, z_*).$$ Consequently, $z_{n+1}<z_*$. Hence, the sequence $(z_n)$ obtained from the scheme \ref{iter} is monotonically increasing and bounded above by $z_*$. Hence, $(z_n)$ will converge to a fixed point $z_*$ of $G(z)$. The proof for the second case $z_*<z_r$ follows analogously. Hence, it is left out.  
	\end{proof}
	
	It is interesting to note that Theorem \ref{thm2}, Theorem \ref{thm5} and Theorem \ref{theorem4} are built on the assumption $\dot{r}(z)\leq0$ in $J'$. In some applications, $\dot{r}(z)\geq0$ is bounded above by a small positive number. The following theorem guarantees the convergence of the sequence $(z_n)$ generated by the proposed iterative method in such situations.
	
	\begin{theorem}\label{t12}
		Let $k_1\in (-\infty, \frac{1}{2})$. Then, the following statements hold:
		\begin{enumerate}
			\item Let $r(z)$ be positive and $\dot{r}(z)<k_1$ in $(z_w, z_*)$. Then, the sequence $(z_n)$ generated by the proposed iterative method converges monotonically to $z_*$ for any initial guess $z_0 \in (z_w, z_*)$.
			\item Let $r(z)$ be negative and $\dot{r}(z)<k_1$ in $(z_*, z_w')$. Then, the sequence $(z_n)$ generated by the proposed iterative method converges monotonically to $z_*$ for any initial guess $z_0 \in (z_*, z_w')$.
		\end{enumerate}
	\end{theorem}
	\begin{proof}
		First, we prove part \textit{(1)}. From the hypothesis $r(z)$ and $h(z)$ are having opposite sign in $(z_w, z_*)$. Hence, the function $G(z)$ is well-defined in $(z_w, z_*)$. Moreover, for $z \in (z_w, z_*)$, we have 
		\begin{align*}
			\dot{G}(z)&\geq\dfrac{2h^2-4\dot{r}(z)h^2(z)}{(2+h^2(z)-2r(z)h(z))^2}\geq\dfrac{h^2(z)(2-4k_1)}{2+h^2(z)-2r(z)h(z))^2}\geq0.
		\end{align*}
		Let $z_0 \in (z_w, z_*)$. From the definition of $z_1$, one can conclude that $z_0<z_1$. Now, $$z_1-z_*=G(z_0)-G(z_*)=\dot{G}(c)(z_0-z_*) \hspace{3mm}\text{for some}\hspace{2mm}c \in (z_0, z_*).$$ Consequently, $z_1<z_*$. Consider this conclusion is valid for $k=1,2,\dots,n$. i.e., $z_1,z_2,\dots,z_n$ exist and $z_{w}<z_0<z_1<\dots<z_n<z_{*}$. Now,  $z_{n+1}=G(z_n)=z_n-\dfrac{2h(z_n)}{2+h^2(z_n)-2r(z_n)h(z_n)}$ is well-defined and it is easy to conclude that $z_n<z_{n+1}$.
		\begin{equation*}
			z_{*}-z_{n+1}=G(z_{*})-G(z_n)=\dot{G}(c)(z_{*}-z_n)>0, \hspace{2mm}\text{for some}\hspace{2mm}c \in (z_n,z_{*}).
		\end{equation*}
		Consequently, $z_{n+1}<z_{*}$. Thus, the sequence $(z_n)$ generated by the scheme (\ref{iter}) is monotone increasing and bounded above by $z_{*}$. By using \eqref{iter}, it is easy to verify that $(z_n)$ will converge to the zero $z_{*}$ of $h(z)$. The proof of part \textit{(2)} is analogous, hence it is left out.
	\end{proof}
	Note that the above theorem does not guarantee the convergence of the sequence if $r(z)\geq 0$, $\dot{r}(z)<k_1$ and $z_0 \in (z_*,z_w')$. Similarly, there is no information regarding the convergence if $r(z)\leq0$, $\dot{r}(z)<k_1$ and $z_0 \in (z_w, z_*)$. We can answer this question favourably with additional assumptions on $r(z)$.
	\begin{theorem}\label{t11}
		Let $k_1$ and $k_2$ be two constants such that $0<k_2\leq 1$ and $8k_2^2+6k_1-3<0.$ Then, the following statements hold.
		\begin{enumerate}
			\item Let $0<r(z)<k_2$ and $\dot{r}(z)<k_1$ in $(z_*, z_w')$. Then, the sequence $(z_n)$ generated by the proposed iterative method \eqref{iter} converges monotonically to $z_*$ for any initial guess $z_0 \in (z_*, z_w')$.
			\item Let $-k_2<r(z)<0$ and $\dot{r}(z)<k_1$ in $(z_w, z_*)$. Then, the sequence $(z_n)$ generated by the proposed iterative method \eqref{iter} converges monotonically to $z_*$ for any initial guess $z_0 \in (z_w, z_*)$.
		\end{enumerate}
	\end{theorem}
	\begin{proof}
		First, we prove part \textit{(1)}. As $0<k_2\leq1$, $G(z)$ is well-defined for all $z$ in $J'$. Moreover, we have 
		\begin{align*}
			\dot{G}(z)\geq\dfrac{3h^4+2h^2-4\dot{r}(z)h^2(z)-8r(z)h^3(z)}{(2+h^2(z)-2r(z)h(z))^2}
			\geq\dfrac{h^2(z)(3h^2(z)-8k_2h(z)+(2-4k_1))}{(2+h^2(z)-2r(z)h(z))^2}.
		\end{align*}
		Now, consider the polynomial $\mathrm{f}(h)=3h^2(z)-8k_2h(z)+(2-4k_1)$. The discriminant $D$ of the corresponding polynomial $\mathrm{f}(h)$ is $8(8k_2^2+6k_1-3)<0$. Consequently, $\mathrm{f}(h)\geq0$ in $(z_*, z_w')$. Hence, $\dot{G}(z)\geq 0$ in $(z_*, z_w')$. Let $z_0 \in (z_*, z_w')$. From the definition of $z_1$, it is easy to see that $z_1<z_0$. Now, $$z_1-z_*=G(z_0)-G(z_*)=\dot{G}(c)(z_0-z_*)\hspace{3mm}\text{for some}\hspace{2mm}c \in (z_*, z_0).$$ Hence, $z_*<z_1$. Consider this conclusion is valid for $k=1,2,\dots,n.$ i.e., $z_1, z_2, \dots, z_n$ exist and $z_*<z_n<z_{n-1}<\dots<z_1<z_w'$. Now, $z_{n+1}=G(z_n)$ is well-defined and $z_{n+1}<z_n$. 
		$$z_{n+1}-z_{*}=G(z_n)-G(z_*)=\dot{G}(c)(z_n-z_*)\hspace{3mm}\text{for some}\hspace{2mm}c \in (z_*, z_n).$$ Consequently, $z_*<z_{n+1}$. Thus, the sequence $(z_n)$ generated by the scheme \ref{iter} is monotone decreasing and bounded below by $z_*$. By using \ref{iter}, $(z_n)$ converges to the zero $z_*$ of $h(z)$. The proof of part \textit{(2)} is analogous, hence it is left out.		
	\end{proof}
	By clubbing the conclusions from Theorem \ref{t12} and Theorem \ref{t11}, we can get the global convergence property of the proposed iterative method in the interval $J'$ even when $\dot{r}(z)\geq 0$. This observation is stated as a remark.
	
	\begin{remark}\label{r4}
		Let $k_1$ and $k_2$ be two constants such that $0<k_2\leq1$ and $8k_2^2+6k_1-3<0$. Let $r(z)\neq0$, $\dot{r}(z)<k_1$ and $\left|r(z)\right| <k_2$ in $J'$. Then, the sequence $(z_n)$ generated by the proposed iterative method \eqref{iter} converges monotonically to $z_*$ for any initial guess $z_0 \in J'$. It is interesting to note that the condition $8k_2^2+6k_1-3<0$ naturally ensures that $k_1<\frac{1}{2}$.
	\end{remark}
	
	The following theorem provides sufficient conditions for monotone convergence of the Newton method in the interval $J'$.
	\begin{theorem}\label{Newton}
		Let $r(z)$ be a non-increasing function on $J'$.
		\begin{enumerate}
			\item If $r(z)$ is positive in $(z_w, z_*)$, then the sequence $(z_n)$ generated by the Newton method \eqref{iter1} converges monotonically to $z_*$ for any initial guess $z_0 \in (z_w, z_*)$. 
			\item If $r(z)$ is negative in $(z_*, z_w')$, then the sequence $(z_n)$ generated by the Newton method \eqref{iter1} converges monotonically to $z_*$ for any initial guess $z_0 \in (z_*, z_w')$.
		\end{enumerate}
	\end{theorem}	
	
	\begin{proof}
		The proof is analogous to Theorem \ref{thm2}, and hence it is left out. 
	\end{proof}

    \subsection*{Finding All The Zeros in a Given Interval}
    The proposed third-order method exhibits global convergence between successive singular points, similar to the second-order iterative method \cite{2003, 2002, 2010}, under suitable assumptions. In \cite{2003, 2002}, using Lemma \ref{lemma2} and Lemma \ref{lemma3}, by adding or subtracting $\frac{\pi}{2}$ with the current zero of $y(x)$, the initial guess for the next zero of $y(x)$ is obtained. 
    As the proposed iterative method is derived from the same Riccati differential equation in \cite{2003}, the initial guesses in \cite{2003} discussed for the second-order iterative method will work perfectly for the proposed third-order iterative method in conjunction with the various convergence conditions discussed in Theorem \ref{thm2} to Theorem \ref{t11}. Table \ref{on I' nonzero r}, Table \ref{on I' when r has zero} and Table \ref{initial guess on [a', b']} present the various cases in finding the initial guess for the proposed third-order iterative method.

    Let $z_{y_i}$ and $z_{w_i}$, $i \in \mathbb{N}$ denote the $i^{th}$ zero of $y(z)$ and $w(z)$, respectively. Let $I'=(z_{w_i}, z_{w_{i+k+1}}), i\in \mathbb{N}, k\in \mathbb{N}\cup\{0\}$. Let $z_{y_i}, z_{y_{i+1}}, \dots, z_{y_{i+k}}$ be the zeros of $y(z)$ in $I'$. Let $z_0^{j}$ denote the initial guess to find the $j^{th}$ zero of $y(z)$.  Table \ref{on I' nonzero r} presents the initial guess for the third-order iterative method if the interval $I'$ doesn't contain the zero of $r(z)$.  Figure \ref{fig1} and Figure \ref{fig2} illustrate this scenario.

\begin{table}[h]
\renewcommand{\arraystretch}{1.35}
\begin{tabular}{|c|c|c|}
\hline
\textbf{\begin{tabular}[c]{@{}c@{}}Convergence \\ conditions for \eqref{iter}\end{tabular}}
& \textbf{Sign of $r(z)$}
& \textbf{\begin{tabular}[c]{@{}c@{}}Initial guesses for computing \\ all zeros in $I'$\end{tabular}} \\
\hline
\multirow{2}{*}{\begin{tabular}[c]{@{}c@{}}
$\dot{r}(z) \leq 0$ and $\lvert r(z)\rvert< 1$ on $I'$\\
or\\
$0\leq \dot{r}(z)<k_1$ and $\left|r(z)\right|<k_2$ on $I'$,\\
where $8k_2^2+6k_1-3<0,$ \& $0<k_2\leq 1$
\end{tabular}}
& $r(z)>0$ on $I'$
& \begin{tabular}[c]{@{}c@{}}
$z_0^{i+k}=z_{w_{i+k+1}}-\dfrac{\pi}{2}\in (z_{y_{i+k}}, z_{w_{i+k+1}})$\\
$z_0^{j}=z_{y_{j+1}}-\dfrac{\pi}{2}\in (z_{y_j}, z_{w_{j+1}}),$ $i\leq j<i+k$
\end{tabular} \\
\cline{2-3}
& $r(z)<0$ on $I'$
& \begin{tabular}[c]{@{}c@{}}
$z_0^{i}=z_{w_i}+\dfrac{\pi}{2}\in (z_{w_i}, z_{y_i})$\\
$z_0^{j}=z_{y_{j-1}}+\dfrac{\pi}{2}\in (z_{w_j}, z_{y_j}),$ $i<j\leq i+k$
\end{tabular} \\
\hline
\end{tabular}
\caption{Initial guesses for finding all zeros in the interval $I'$ when $r(z) \neq 0$}
\label{on I' nonzero r}
\end{table}	
Table \ref{on I' when r has zero} presents the initial guess for the third-order iterative method if the interval $I'$ contains a unique zero of $r(z)$. Figure \ref{fig3} and Figure \ref{fig4} illustrate this scenario.
	\begin{figure}[h]
		\centering
		\begin{tikzpicture}
			
			\draw[thick] (0,0) -- (12,0);

			\draw[fill=red] (0,0) circle (2pt); 
			\node[above] at (0,0) {$z_{w_i}$};
			
			\draw[fill=blue] (1.5,0) circle (2pt); 
			\node[above] at (1.5,0) {$z_{y_i}$};
			
			\draw[fill=green] (2.5,0) circle (2pt);
			\node[below,scale=0.7] at (2.5,0) {$z_{y_{i+1}} - \frac{\pi}{2}$}; 
			
			\draw[fill=red] (3,0) circle (2pt); 
			\node[above] at (3,0) {$z_{w_{i+1}}$};
			
			\draw[fill=blue] (4.5,0) circle (2pt); 
			\node[above] at (4.5,0) {$z_{y_{i+1}}$};
			
			\node at (6,0.05) {$\dots$};

			\draw[fill=blue] (7.5,0) circle (2pt); 
			\node[above] at (7.5,0) {$z_{y_{i+k-1}}$};

			\draw[fill=green] (8.5,0) circle (2pt); 
			\node[below,scale=0.7] at (8.5,0) {$z_{y_{i+k}} - \frac{\pi}{2}$};

			\draw[fill=red] (9,0) circle (2pt); 
			\node[above] at (9,0) {$z_{w_{i+k}}$};

			\draw[fill=blue] (10.5,0) circle (2pt); 
			\node[above] at (10.5,0) {$z_{y_{i+k}}$};

			\draw[fill=green] (11.5,0) circle (2pt); 
			\node[below,scale=0.7] at (11.5,0) {$z_{w_{i+k+1}} - \frac{\pi}{2}$};

			\draw[fill=red] (12,0) circle (2pt); 
			\node[above] at (12,0) {$z_{w_{i+k+1}}$};

			\draw[|->, thick] (11.5,-0.5) -- (10.5,-0.5);

			\draw[->, thick] (10.5, 0.02) arc[start angle=0,end angle=180,radius=1]; 
			\draw[<-, thick, , >=latex] (9.5, 1) -- ++(0.2, 0);
			\draw[->, thick] (7.5, 0.02) arc[start angle=0,end angle=180,radius=1];  
			\draw[<-, thick,  >=latex] (6.5, 1) -- ++(0.2, 0); 
			\draw[->, thick] (4.5, 0.02) arc[start angle=0,end angle=180,radius=1]; 
			\draw[<-, thick,  >=latex] (3.5, 1) -- ++(0.2, 0); 
			\node[draw, fill=white, anchor=north west, rounded corners] at (-0.5, 1.5) {
				\scalebox{0.6}{	\begin{tabular}{rl}
						\textcolor{red}{\textbullet} & Singularities of $h(z)$\\
						\textcolor{blue}{\textbullet} & Zeros of $h(z)$\\
						\textcolor{green}{\textbullet} & Initial Guess \\
				\end{tabular} }
			};
		\end{tikzpicture}
		\caption{Illustration of Table \ref{on I' nonzero r} for the case $r(z)>0$}
		\label{fig1}
	\end{figure}

		\begin{figure}[H]
		\centering
		\begin{tikzpicture} 
			\draw[thick] (0,0) -- (12,0);
			
			\foreach \x/\label in {0/$z_{w_i}$, 3/$z_{w_{i+1}}$} {
				\draw[fill=red] (\x,0) circle (2pt); 
				\node[above] at (\x,0) {\label}; 
			}
            \draw[fill=blue] (4.5,0) circle (2pt); 
			\node[above] at (4.5,0) {$z_{y_i+1}$};
			\draw[fill=blue] (1.5,0) circle (2pt); 
			\node[above] at (1.5,0) {$z_{y_i}$};

			\draw[fill=green] (0.65,0) circle (2pt); 
			\node[below,scale=0.7] at (0.65,0) {$z_{w_i} + \frac{\pi}{2}$};

			\draw[fill=green] (3.5,0) circle (2pt); 
			\node[below,scale=0.7] at (3.5,0) {$z_{y_i} + \frac{\pi}{2}$}; 
			
			\draw[fill=blue] (7.6,0) circle (2pt); 
			\node[above] at (7.6,0) {$z_{y_{i+k-1}}$};

			\draw[fill=red] (9,0) circle (2pt); 
			\node[above] at (9,0) {$z_{w_{i+k}}$};
			
			\draw[fill=blue] (10.5,0) circle (2pt); 
			\node[above] at (10.5,0) {$z_{y_{i+k}}$};
			
			\draw[fill=red] (12,0) circle (2pt); 
			\node[above] at (12,0) {$z_{w_{i+k+1}}$};

			\draw[fill=green] (9.6,0) circle (2pt); 
			\node[below,scale=0.7] at (9.6,0) {$z_{y_{i+k-1}} + \frac{\pi}{2}$};

			\draw[<-, thick] (3.5,0) arc[start angle=0,end angle=180,radius=1]; 
	
			\draw[->, >=latex, thick] (2.5, 1) -- ++(0.2, 0); 
			
			\draw[<-, thick] (6.5,0) arc[start angle=0,end angle=180,radius=1]; 
			\draw[->, >=latex, thick] (5.5, 1) -- ++(0.2, 0); 
			
			\draw[<-, thick] (9.6,0) arc[start angle=0,end angle=180,radius=1]; 
			\draw[->, >=latex, thick] (8.6, 1) -- ++(0.2, 0); 
			\node at (7, 0.05) {$\dots$};

			\draw[|->, thick] (0.65,-0.5) -- (1.5,-0.5);

			\node[draw, fill=white, anchor=north east, rounded corners] at (13, 1.5) {
				\scalebox{0.6}{	\begin{tabular}{rl}
						\textcolor{red}{\textbullet} & Singularities of $h(z)$\\
						\textcolor{blue}{\textbullet} & Zeros of $h(z)$\\
						\textcolor{green}{\textbullet} & Initial Guess \\
				\end{tabular} }
			};
			
		\end{tikzpicture}
		\caption{Illustration of Table \ref{on I' nonzero r} for the case $r(z)<0$}
		\label{fig2}
	\end{figure}

\begin{table}[h]
\renewcommand{\arraystretch}{1.35}
\begin{tabular}{|c|c|c|}
\hline
\textbf{\begin{tabular}[c]{@{}c@{}}Convergence \\ conditions for \eqref{iter} \end{tabular}}
& \textbf{Location of $z_r$}
& \textbf{\begin{tabular}[c]{@{}c@{}}Initial guesses for computing \\ all zeros in $I'$\end{tabular}} \\
\hline
\multirow{2}{*}{\begin{tabular}[c]{@{}c@{}}
$\dot{r}(z) \leq 0$, $\lvert r(z) \rvert< 1$ on $I'$\\
and $z_r \in I'$, $r(z_r) = 0$
\end{tabular}}
& $z_{w_j}<z_{y_j}<z_r<z_{w_{j+1}}$
& \begin{tabular}[c]{@{}c@{}}
$z_0^{j} = z_r$, $z_0^{j+1}=z_r+\dfrac{\pi}{2} \in (z_{w_{j+1}}, z_{y_{j+1}})$\\
$z_0^{k}=z_{y_{k+1}}-\dfrac{\pi}{2} \in (z_{y_k}, z_{w_{k+1}}),\quad i\leq k<j$\\
$z_0^{t}=z_{y_{t-1}}+\dfrac{\pi}{2} \in (z_{w_t}, z_{y_t}),\quad j+1<t\leq i+k$
\end{tabular} \\
\cline{2-3}
& $z_{w_j}<z_r<z_{y_j}<z_{w_{j+1}}$
& \begin{tabular}[c]{@{}c@{}}
$z_0^{j} = z_r$, $z_0^{j-1}=z_r-\dfrac{\pi}{2}\in (z_{y_{j-1}}, z_{w_j})$\\
$z_0^{k}=z_{y_{k+1}}-\dfrac{\pi}{2}\in (z_{y_k}, z_{w_{k+1}}),\quad i\leq k<j-1$\\
$z_0^{t}=z_{y_{t-1}}+\dfrac{\pi}{2} \in (z_{w_t}, z_{y_t}),\quad j<t\leq i+k$
\end{tabular} \\
\hline
\end{tabular}
\caption{Initial guesses for finding all zeros in the interval $I'$ when $r(z)$ has a zero}
\label{on I' when r has zero}
\end{table}

	\begin{figure}[h]
		\centering
		\begin{tikzpicture}

			\draw[thick] (0,0) -- (13,0);

			\draw[fill=red] (0,0) circle (2pt);
			\node[above] at (0,0) {$z_{w_i}$};
			
			\node at (1,0.05) {$\dots$}; 
			
			\draw[fill=green] (2,0) circle (2pt); 
			\node[below,scale=0.8] at (2,0) {$z_{y_{j-1}} - \frac{\pi}{2}$};
			
			\draw[fill=red] (3,0) circle (2pt); 
			\node[above] at (3,0) {$z_{w_{j-1}}$};
			
			\draw[fill=blue] (4,0) circle (2pt); 
			\node[above] at (4,0) {$z_{y_{j-1}}$};
			
			\draw[fill=green] (5,0) circle (2pt); 
			\node[below,scale=0.8] at (5,0) {$z_{y_j} - \frac{\pi}{2}$};
			
			\draw[fill=red] (6,0) circle (2pt); 
			\node[above] at (6,0) {$z_{w_j}$};
			
			\draw[fill=blue] (7,0) circle (2pt); 
			\node[above,scale=0.8] at (7,0) {$z_{y_j}$};
			
			\draw[fill=green] (8,0) circle (2pt);
			\node[below] at (8,0) {$z_{r}$};
			
			\draw[fill=red] (9,0) circle (2pt);
			\node[above] at (9,0) {$z_{w_{j+1}}$};
			
			\draw[fill=green] (10,0) circle (2pt); 
			\node[below,scale=0.8] at (10,0) {$z_{r} + \frac{\pi}{2}$};
			
			\draw[fill=blue] (11,0) circle (2pt); 
			\node[above] at (11,0) {$z_{y_{j+1}}$};
			
			\node at (12,0.05) {$\dots$}; 
			
			\draw[fill=red] (13,0) circle (2pt); 
			\node[above] at (13,0) {$z_{w_{i+k+1}}$};

			\draw[|->, thick] (8,-0.5) -- (7,-0.5); 
			\draw[|->, thick] (10,-0.5) -- (11,-0.5);

			\draw[->, thick] (4,0.02) arc[start angle=0,end angle=180,radius=1]; 
			\draw[<-, thick,  >=latex] (3.0, 1) -- ++(0.2, 0);
			\draw[->, thick] (7,0.02) arc[start angle=0,end angle=180,radius=1]; 
			\draw[<-, thick,  >=latex] (6, 1) -- ++(0.2, 0); 
			\draw[<-, thick] (10,0.02) arc[start angle=0,end angle=180,radius=1]; 
			\draw[->, >=latex, thick] (9, 1) -- ++(0.2, 0); 
			
			\node[draw, fill=white, anchor=north east, rounded corners] at (13.5, 1.5) {
				\scalebox{0.6}{	\begin{tabular}{rl}
						\textcolor{red}{\textbullet} & Singularities of $h(z)$\\
						\textcolor{blue}{\textbullet} & Zeros of $h(z)$\\
						\textcolor{green}{\textbullet} & Initial Guess \\
				\end{tabular} }
			};
			
		\end{tikzpicture}
		\caption{Illustration of Table \ref{on I' when r has zero} for the case $z_{w_j}<z_{y_j}<z_r<z_{w_{j+1}}$}
		\label{fig3}
	\end{figure}
	
	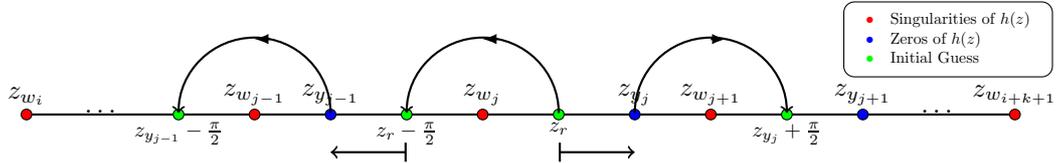
\begin{figure}[h]
		\centering
		\begin{tikzpicture}
			
			\draw[thick] (0,0) -- (13,0);

			\draw[fill=red] (0,0) circle (2pt); 
			\node[above] at (0,0) {$z_{w_i}$};
			
			\node at (1,0.05) {$\dots$};
			
			\draw[fill=green] (2,0) circle (2pt); 
			\node[below,scale=0.8] at (2,0) {$z_{y_{j-1}} - \frac{\pi}{2}$};
			
			\draw[fill=red] (3,0) circle (2pt); 
			\node[above] at (3,0) {$z_{w_{j-1}}$};
			
			\draw[fill=blue] (4,0) circle (2pt);
			\node[above] at (4,0) {$z_{y_{j-1}}$};
			
			\draw[fill=green] (5,0) circle (2pt); 
			\node[below,scale=0.8] at (5,0) {$z_{r} - \frac{\pi}{2}$};
			
			\draw[fill=red] (6,0) circle (2pt); 
			\node[above] at (6,0) {$z_{w_j}$};
			
			\draw[fill=green] (7,0) circle (2pt); 
			\node[below,scale=0.8] at (7,0) {$z_r$};
			
			\draw[fill=blue] (8,0) circle (2pt); 
			\node[above] at (8,0) {$z_{y_j}$};
			
			\draw[fill=red] (9,0) circle (2pt); 
			\node[above] at (9,0) {$z_{w_{j+1}}$};
			
			\draw[fill=green] (10,0) circle (2pt); 
			\node[below,scale=0.8] at (10,0) {$z_{y_j} + \frac{\pi}{2}$};
			
			\draw[fill=blue] (11,0) circle (2pt); 
			\node[above] at (11,0) {$z_{y_{j+1}}$};
			
			\node at (12,0.05) {$\dots$};
			
			\draw[fill=red] (13,0) circle (2pt); 
			\node[above] at (13,0) {$z_{w_{i+k+1}}$};

			\draw[|->, thick] (5,-0.5) -- (4,-0.5); 
			\draw[|->, thick] (7,-0.5) -- (8,-0.5); 
		
			\draw[->, thick] (4,0.02) arc[start angle=0,end angle=180,radius=1]; 
			\draw[<-, thick,  >=latex] (3.0, 1) -- ++(0.2, 0);
			\draw[->, thick] (7,0.02) arc[start angle=0,end angle=180,radius=1]; 
			\draw[<-, thick,  >=latex] (6, 1) -- ++(0.2, 0); 
			\draw[<-, thick] (10,0.02) arc[start angle=0,end angle=180,radius=1];
			\draw[->, >=latex, thick] (9, 1) -- ++(0.2, 0); 
			
			\node[draw, fill=white, anchor=north east, rounded corners] at (13.5, 1.5) {
				\scalebox{0.6}{	\begin{tabular}{rl}
						\textcolor{red}{\textbullet} & Singularities of $h(z)$\\
						\textcolor{blue}{\textbullet} & Zeros of $h(z)$\\
						\textcolor{green}{\textbullet} & Initial Guess \\
				\end{tabular} }
			};
			
		\end{tikzpicture}
		\caption{Illustration of Table \ref{on I' when r has zero} for the case $z_{w_j}<z_r<z_{y_j}<z_{w_{j+1}}$}
		\label{fig4}
	\end{figure}

     In many situations, we have an interval $[a', b']$ containing a finite number of zeros of $h(z)$, where $a'$, $b'$ need not be a singular point of $h(z)$. Table \ref{initial guess on [a', b']} provides suitable initial guesses to find all the zeros in a given interval $[a', b']$. Assume that $z_{y_j} \in [a', b'],$ $\hspace{3mm}\forall\hspace{2mm}j=i, i+1, \dots, i+k,$ $i\in \mathbb{N},$ $k\in \mathbb{N}\cup\{0\}$ and $z_{y_j}<z_{y_k}$ if $j<k$. Define
\begin{equation*}
\xi_1(x)=
\begin{cases}
1, & x\ge 0,\\
-1, & x<0,
\end{cases}
\qquad \text{and} \qquad
\xi_2(x)=
\begin{cases}
1, & x>0,\\
-1, & x\le 0.
\end{cases}
\end{equation*}


\begin{table}[h]
\renewcommand{\arraystretch}{1.35}
\begin{tabular}{|c|c|c|}
\hline
\textbf{\begin{tabular}[c]{@{}c@{}}Convergence \\ conditions for \eqref{iter} \end{tabular}}
& \textbf{Sign of $r(z)$}
& \textbf{\begin{tabular}[c]{@{}c@{}}Initial guesses for computing \\ all zeros in $[a', b']$\end{tabular}} \\
\hline
\multirow{2}{*}{\begin{tabular}[c]{@{}c@{}}
$\dot{r}(z) \leq 0$ and $\lvert r(z)\rvert<1$ on $[a', b']$\\
or\\
$0\leq \dot{r}(z)<k_1$ and $\left|r(z)\right|<k_2$ on $[a', b']$,\\
where $8k_2^2+6k_1-3<0$, \& $0<k_2\leq 1$
\end{tabular}}
& $r(z)>0$ on $[a', b']$
& \begin{tabular}[c]{@{}c@{}}
$z_0^{i+k}=b'-\dfrac{\pi}{2}
\left(\dfrac{1-\xi_1(h(b'))}{2}\right)$\\
$z_0^{j}=z_{y_{j+1}}-\dfrac{\pi}{2},\quad i\leq j\leq i+k-1$
\end{tabular} \\
\cline{2-3}
& $r(z)<0$ on $[a', b']$
& \begin{tabular}[c]{@{}c@{}}
$z_0^{i}=a'+\dfrac{\pi}{2}
\left(\dfrac{1+\xi_2(h(a'))}{2}\right)$\\
$z_0^{j}=z_{y_{j-1}}+\dfrac{\pi}{2},\quad i+1\leq j\leq i+k$
\end{tabular} \\
\hline
\end{tabular}
\caption{Initial guesses for finding all zeros in $[a', b']$}
\label{initial guess on [a', b']}
\end{table}

	\begin{remark}\label{remark5}
    Let $\Omega(z)=1+\dot{r}(z)-r^2(z)$ as in Remark \ref{remark_prelim_2}. Using the Sturm Convexity Theorem \ref{sturm_lemma}, one can improve the initial guesses provided in Table \ref{on I' nonzero r} to Table \ref{initial guess on [a', b']} as in \cite[p 832]{2003}.

		\begin{enumerate}
			\item Let $r(z)>0$ and $\dot{\Omega}(z)>0$ in $I'$. Define $\Delta_{B_j}=z_{y_{j+2}}-z_{y_{j+1}}$, $i \leq j \leq i+k-2$. Then, 
			\begin{equation}\label{concave}
				z_{y_{j}}<z_{y_{j+1}}-\Delta_{B_j}<z_{y_{j+1}}-\frac{\pi}{2}<z_{w_{j+1}}.
			\end{equation}
			\item Let $r(z)<0$ and $\dot{\Omega}(z)<0$ in $I'$. Define $\Delta_{F_j}=z_{y_{j-1}}-z_{y_{j-2}}$, $i+2\leq j\leq i+k$. Then,
			\begin{equation}\label{convex}
				z_{w_j}<z_{y_{j-1}}+\frac{\pi}{2}<z_{y_{j-1}}+\Delta_{F_j}<z_{y_j}.
			\end{equation}
		\end{enumerate}
	\end{remark}

	\section{Application} \label{application}
	
	In this section, the versatile application of the proposed third-order iterative method is demonstrated by finding all the zeros of orthogonal polynomials, such as Hermite and Legendre, and special functions, including the Bessel and cylinder functions, within a given interval.  
	
	\subsection{Legendre polynomial}\label{Legendre}
	In this subsection, all the zeros of a given Legendre polynomial are obtained using the proposed third-order iterative method. It is well-known that the Legendre polynomial $P_n(x)$ of degree $n \in \mathbb{N}$, is a solution of the second order linear differential equation $$(1-x^2)y''(x)-2xy'(x)+n(n+1)y(x)=0.$$
	By using the formulas in \cite[p. 147]{speical functions}, \cite[p. 166]{speical functions} and three-term recurrence (see \cite[p. 148]{speical functions}), one can get the following CDE
	\begin{equation}\label{leg_rec}
		\left\{\begin{aligned}
			P'_n(x)&=\frac{(n+1)x}{1-x^2}P_{n}(x)-\frac{n+1}{1-x^2}P_{n+1}(x)\\
			P'_{n+1}(x)&=-\frac{(n+1)x}{1-x^2}P_{n+1}(x)+\frac{n+1}{1-x^2}P_{n}(x).
		\end{aligned}\right.
	\end{equation}
	For the choice $y(x)=P_n(x)$ and $w(x)=P_{n+1}(x)$ all the conclusions of Theorem \ref{thm1} hold true in $(-1, 1)$. One can check that $z(x)=\dfrac{n+1}{2}\log\left(\dfrac{1+x}{1-x}\right)$ and $x(z)=\tanh\left(\dfrac{z}{n+1}\right)$ are the transformations discussed in Section \ref{2}. Hence, $h(z)=\frac{\tilde{y}(x(z))}{\tilde{w}(x(z))}=-\dfrac{P_n(x(z))}{P_{n+1}(x(z))}$ satisfies the Riccati differential equation \eqref{ricatti} with $r(z)=-\tanh\left(\dfrac{z}{n+1}\right)$. Hence, $z \mapsto r(z)$ is decreasing and $\left| r(z) \right|<1$ for all $z$ in $(-\infty, \infty)$. As the zeros of $P_n(x)$ are distributed symmetrically about $x=0$ (see \cite[eq. 6.20, p. 143]{speical functions}), it is enough to find all the positive zeros of $P_n(x)$. If $n$ is even, then $z=0$ is a singular point of $h(z)$. Hence, using initial guesses from Table \ref{on I' nonzero r}, all $\frac{n}{2}$ zeros in $(0, \infty)$ can be obtained. i.e. By choosing $z_0^{1}=\frac{\pi}{2}$ as the initial guess, one can find the first positive zero $z_{y_1}$ of $P_n(x(z))$ then rest of the zeros can be obtained successively using the initial guess $z_0^{i}=z_{y_{i-1}}+\frac{\pi}{2}$, $i=2,3,\cdots \frac{n}{2}$. 
	
	If $n$ is odd then $z=0$ is the first zero of $h(z)$ in $[0, \infty)$. Hence, using initial guesses from Table \ref{on I' nonzero r}, all $\frac{(n-1)}{2}$ zeros in $(0, \infty)$ can be obtained. i.e. By choosing $z_0^{1}=\frac{\pi}{2}$ as the initial guess, one can find the first positive zero $z_{y_1}$ of $P_n(x(z))$ then rest of the zeros can be obtained successively using the initial guess $z_0^{i}=z_{y_{i-1}}+\frac{\pi}{2}$, $i=2,3,\cdots \frac{n-1}{2}$. 
	
	Note that $\Omega(z)=1+\dot{r}(z)-r^2(z)=1-\dfrac{1}{n+1}\text{sech}^2\left(\frac{z}{n+1}\right)-\tanh^2\left(\frac{z}{n+1}\right)$. Hence, for $n \in \mathbb{N},$ $z \mapsto \Omega(z)$ is decreasing on $(0, \infty)$. Consequently, using Remark \ref{remark5} and \eqref{convex}, one can get a better initial guess as discussed.

	\subsection{Hermite polynomial} \label{HP-T} In this subsection, all the zeros of a given Hermite polynomial are obtained using the proposed third-order iterative method. It is well-known that the Hermite polynomial $H_n(x)$ of degree $n \in \mathbb{N}$, is a solution of the second order linear differential equation (see \cite[eq. 6.34, p. 150]{speical functions}) $$y''(x)-2xy'(x)+2ny(x)=0.$$ By using the formula \cite[eq. 18.9.25]{OLB10} and the recurrence relation \cite[Table 18.9.1]{OLB10}, one can get the following CDE
	
	\begin{equation}\label{rec_H}
		\left\{ \begin{aligned}
			H'_{n}(x)&=2xH_{n}(x)-H_{n+1}(x)\\
			H'_{n+1}(x)&=2(n+1)H_{n}(x).
		\end{aligned} \right.
	\end{equation}
	For the choice $y(x)=H_n(x)$ and $w(x)=H_{n+1}(x)$ all the conclusions of Theorem \ref{thm1} hold true in $(-\infty, \infty)$. One can verify that $z(x)=\sqrt{2(n+1)}x$ and $x(z)=\frac{z}{\sqrt{2(n+1)}}$ are the transformations discussed in Section \ref{2}. Hence, $h(z)=\frac{\tilde{y}(x(z))}{\tilde{w}(x(z))}=-\frac{\sqrt{2(n+1)}H_n(x(z))}{H_{n+1}(x(z))}$ satisfies the Riccati differential equation \eqref{ricatti}, with $r(z)=-\frac{z}{2(n+1)}$. Hence, $z\mapsto r(z)$ is decreasing on $(-\infty, \infty)$. Using the bounds for zeros of the Hermite polynomial (see \cite[p 168]{speical functions}), one can conclude that all zeros of the Hermite polynomial $H_n(x)$ must lie in the interval $(-\sqrt{2n+1}, \sqrt{2n+1})$. Consequently, all zeros of $H_n(x(z))$ lie in the interval $\mathcal{I}$, where \linebreak $\mathcal{I}=(-\sqrt{2(n+1)}\sqrt{2n+1}, \sqrt{2(n+1)}\sqrt{2n+1})$. Furthermore, one can verify that $\left|r(z)\right|<1$, $\hspace{2mm}\forall\hspace{2mm}z \in (-2(n+1), 2(n+1))$. Moreover, the interval $(-2(n+1), 2(n+1))$ contains all zeros of $H_n(x(z))$. As the zeros of $H_n(x)$ are distributed symmetrically about $x=0$ (see \cite[Table 18.6.1]{OLB10}), it is enough to find all the positive zeros of $H_n(x(z))$. If $n$ is even, then $z=0$ is a singular point of $h(z)$. Hence, using initial guesses from Table \ref{on I' nonzero r}, all $\frac{n}{2}$ zeros in $(0, \sqrt{2(n+1)}\sqrt{2n+1})$ can be obtained. i.e. By choosing $z_0^{1}=\frac{\pi}{2}$ as the initial guess, one can find the first positive zero $z_{y_1}$ of $H_n(x(z))$ then rest of the zeros can be obtained successively using the initial guess $z_0^{i}=z_{y_{i-1}}+\frac{\pi}{2}$, $i=2,3,\cdots \frac{n}{2}$.

	If $n$ is odd, $z=0$ is the first zero of $h(z)$. Hence, using initial guesses from Table \ref{on I' nonzero r}, all $\frac{n-1}{2}$ zeros in $(0, \sqrt{2(n+1)}\sqrt{2n+1})$ can be obtained. i.e. By choosing $z_0^{1}=\frac{\pi}{2}$ as the initial guess, one can find the first positive zero $z_{y_1}$ of $H_n(x(z))$ then rest of the zeros can be obtained successively using the initial guess $z_0^{i}=z_{y_{i-1}}+\frac{\pi}{2}$, $i=2,3,\cdots \frac{n-1}{2}$. 
	
	Note that 
	$\Omega(z)=1+\dot{r}(z)-r^2(z)=1-\dfrac{1}{2(n+1)}-\dfrac{z^2}{4(n+1)^2}$ and $z \mapsto \Omega(z)$ is decreasing on $(0, \infty)$. Consequently, using Remark \ref{remark5} and \eqref{convex}, one can get a better initial guess as discussed. It is worth mentioning that one can find all the zeros of the following class of Parabolic cylinder function $U(a,x)$ of the form $U\left(-n-\frac{1}{2}, \sqrt{2}x\right)$ as they are closely related to $H_n(x)$. More specifically,  $H_n(x) = 2^{\frac{n}{2}} e^{\frac{x^2}{2}} U\left(-n-\frac{1}{2}, \sqrt{2}x\right)$ (see \cite[eq. 12.7.2]{OLB10}).

	\subsection{Bessel function} \label{Bessel_im}
	In this subsection, using the proposed third-order iterative method, all the positive zeros of the Bessel function $J_{\mu}(x)$ of order $\mu>-1$ in a given interval $[a',b']$ are obtained. For $k\in \mathbb{N}$, $j_{\mu,k}$ denotes the $k$-th positive zero of $J_\mu(x)$. It is interesting to note that when $\mu>-1$ the zeros of $J_{\mu}(x)$ are all real (see \cite[sec. 10.21]{OLB10}).
	Bessel function $J_{\mu}(x)$ can be considered as a solution of the following second-order linear differential equation $$x^2y''(x)+xy'(x)+(x^2-\mu^2)y(x)=0.$$   
	By using the recurrence relations \cite[eq. 10.6.2]{OLB10}, one can get the following coupled differential equations
	\begin{equation}\label{e1}
		\left\{
		\begin{aligned}
			J_{\mu}'(x)&=-\frac{\mu}{x}J_{\mu}(x)+J_{\mu-1}(x)\\
			J_{\mu-1}'(x)&=\frac{\mu-1}{x}J_{\mu-1}(x)-J_{\mu}(x)
		\end{aligned}\right.
	\end{equation}
	and 
	\begin{equation}\label{e2}
		\left\{ \begin{aligned}
			J_{\mu}'(x)&=\frac{\mu}{x}J_{\mu}(x)-J_{\mu+1}(x)\\
			J_{\mu+1}'(x)&=-\frac{\mu+1}{x}J_{\mu+1}(x)+J_{\mu}(x).
		\end{aligned}\right.
	\end{equation}
	
	\textbf{Case 1 $\mu\geq\frac{1}{2}$:} For the choice $y(x)=J_{\mu}(x)$, $w(x)=J_{\mu-1}(x)$ and using the CDE \eqref{e1}, one can verify that all the conclusions of Theorem \ref{thm1} hold true in $(0, \infty)$. Consequently, $h(x)=\frac{J_{\mu}(x)}{J_{\mu-1}(x)}$ satisfies the Riccati differential equation \eqref{ricatti}, with $r(x)=\frac{\mu-\frac{1}{2}}{x}$. For $k \in \mathbb{N}$, using the inequality (see \cite[eq. 1, p. 2]{bound}) one can get $j_{\mu,k}>\mu$. In other words, all the positive zeros of $J_\mu(x)$ are in $(\mu,\infty)$. Note that for $\mu=\frac{1}{2}$, $r(x)\equiv 0$. Consequently, $h(x)=\tan x$. Hence, the positive zeros of $J_{\frac{1}{2}}(x)$ are $n\pi$, $n\in\mathbb{N}$. For $x \in (\mu, \infty)$, and $\mu>\frac{1}{2}$ it is easy to verify that $x \mapsto r(x)$ is decreasing and $0<r(x)<1$. Let 
	$[a',b'] \subset (\mu,\infty)$ contain $k+1$-zeros. i.e., $j_{\mu,l} \in [a',b']$, $i \leq l \leq i+k$ for some $i\in \mathbb{N}$. To find all the zeros in $[a',b']$,  one can get the initial guesses from Table \ref{initial guess on [a', b']}. i.e., the process will start with finding the largest zero in this interval. If $h(b')= 0$, then the largest zero in this interval is $j_{\mu,i+k}=b'$ otherwise the initial guess for $j_{\mu,i+k}$ is $b'-\frac{\pi}{2}\left(\frac{1-sign(h(b'))}{2}\right)$. The rest of the zeros $j_{\mu,l}$ can be obtained using the initial guess $j_{\mu,l+1}-\frac{\pi}{2}$, $i\leq l \leq i+k-1$. Note that $\Omega(x)=1+\dot{r}(x)-r^2(x)=1-\left(\frac{\mu^2-\frac{1}{4}}{x^2}\right)$ and for $\mu\geq\frac{1}{2}$, $x\mapsto \Omega(x)$ is increasing on $(\mu, \infty)$. Consequently, using Remark \ref{remark5} and \eqref{concave}, one can get a better initial guess, as discussed.
	
	\textbf{Case 2 $-1< \mu <\frac{1}{2}$:} For the choice $y(x)=J_{\mu}(x)$, $w(x)=J_{\mu+1}(x)$ and using the CDE \eqref{e2}, one can verify that all conclusions of Theorem \ref{thm1} hold true in $(0, \infty)$.  Consequently, $h(x)=\frac{\tilde{y}(x)}{\tilde{w}(x)}=-\frac{J_{\mu}(x)}{J_{\mu+1}(x)}$ satisfies the Riccati differential equation \eqref{ricatti}, with $r(x)=-\frac{\mu+\frac{1}{2}}{x}$. Note that for $\mu=-\frac{1}{2}$, $r(x)\equiv 0$. Consequently, $h(x)=-\cot x$. Hence, the positive zeros of $J_{-\frac{1}{2}}(x)$ are $(n+\frac{1}{2})\pi$, $n\in\mathbb{N}$.  We will handle the remaining situations in three sub-cases.
	
	\textbf{Sub case 2(a) $0\leq \mu <\frac{1}{2}$:} For the choice $\alpha=0$ in \cite[p. 490]{Wa22}, one can get $j_{\mu, k}>\frac{\pi}{2} \left(\mu+\frac{3}{2}\right)$, $k\in \mathbb{N}$. In other words, all the positive zeros of $J_\mu(x)$ are in $(\frac{\pi}{2}\left(\mu+\frac{3}{2}\right),\infty)$. For $x \in (\frac{\pi}{2}\left(\mu+\frac{3}{2}\right),\infty)$, one can verify that $x \mapsto r(x)$ is negative, increasing, $$r(x)>-\frac{4(\mu+\frac{1}{2})}{3\pi+2\mu\pi}>-\frac{4}{3\pi} ~~~~ \text{and}~~~0< \dot{r}(x)<\frac{16(\mu+\frac{1}{2})}{9\pi^2}<\frac{16}{9\pi^2}.$$ For the choice $k_1=\frac{16}{9\pi^2}$ and $k_2=\frac{4}{3\pi}$, one can verify that $8k_2^2+6k_1-3<0.$  Let 
	$[a',b'] \subset (\frac{\pi}{2}\left(\mu+\frac{3}{2}\right),\infty)$ contain $k+1$-zeros. i.e., $j_{\mu,l} \in [a',b']$, $i \leq l \leq i+k$ for some $i\in \mathbb{N}$. To find all the zeros in $[a',b']$, one can get the initial guesses from Table \ref{initial guess on [a', b']}. i.e., the process will begin by finding the smallest zero within this interval. If $h(a')= 0$, then the smallest zero in this interval is $j_{\mu,i}=a'$ otherwise the initial guess for $j_{\mu,i}$ is $a'+\frac{\pi}{2}\left(\frac{1+sign(h(a'))}{2}\right)$. The rest of the zeros $j_{\mu,l}$ can be obtained using the initial guess $j_{\mu,l-1}+\frac{\pi}{2}$, $i+1\leq l \leq i+k$.

	\textbf{Sub case 2(b) $-\frac{1}{2}< \mu<0$:} 
	For $k \in \mathbb{N}$, using the inequality (see \cite[eq. 2]{c2}, \cite[p. 1313]{boundz}) one can get $j_{\mu,k}>\sqrt{(\mu+1)(\mu+5)}$. In other words, all the positive zeros of $J_\mu(x)$ are in $(\sqrt{(\mu+1)(\mu+5)},\infty)$. For $x \in (\sqrt{(\mu+1)(\mu+5)}, \infty)$, it is easy to verify that $x \mapsto r(x)$ is negative, increasing $$r(x)=-\frac{\mu+\frac{1}{2}}{x}>\frac{-1}{\sqrt{(\mu+1)(\mu+5)}}>-\frac{1}{3}, ~~~~ \text{and}~~~0<\dot{r}(x)=\frac{\mu+\frac{1}{2}}{x^2}<\frac{1}{2(\mu+1)(\mu+5)}<\frac{2}{9}.$$ For the choice $k_1=\frac{2}{9}$ and $k_2=\frac{1}{3}$, one can verify that $8k_2^2+6k_1-3<0.$  Let 
	$[a',b'] \subset (\sqrt{(\mu+1)(\mu+5)},\infty)$ contain $k+1$-zeros. i.e., $j_{\mu,l} \in [a',b']$, $i \leq l \leq i+k$ for some $i\in \mathbb{N}$. To find all the zeros in $[a',b']$, one can get the initial guesses from Table \ref{initial guess on [a', b']}. i.e., the process will begin by finding the smallest zero within this interval. If $h(a')= 0$, then the smallest zero in this interval is $j_{\mu,i}=a'$ otherwise the initial guess for $j_{\mu,i}$ is $a'+\frac{\pi}{2}\left(\frac{1+sign(h(a'))}{2}\right)$. The rest of the zeros $j_{\mu,l}$ can be obtained using the initial guess $j_{\mu,l-1}+\frac{\pi}{2}$, $i+1\leq l \leq i+k$.
	
	\textbf{Sub case 2(c) $-1<\mu<-\frac{1}{2}$:} It is easy to see that for $x\in (-(\mu+\frac{1}{2}), \infty)$, $0<r(x)<1$ and $x\mapsto r(x)$ is decreasing. Let $[a',b'] \subset (-(\mu+\frac{1}{2}),\infty)$ contain $k+1$-zeros. i.e., $j_{\mu,l} \in [a',b']$, $i \leq l \leq i+k$ for some $i\in \mathbb{N}$. To find all the zeros in $[a',b']$,  one can get the initial guesses from Table \ref{initial guess on [a', b']}. i.e., the process will start with finding the largest zero in this interval. If $h(b')= 0$, then the largest zero in this interval is $j_{\mu,i+k}=b'$ otherwise the initial guess for $j_{\mu,i+k}$ is $b'-\frac{\pi}{2}\left(\frac{1-sign(h(b'))}{2}\right)$. The rest of the zeros $j_{\mu,l}$ can be obtained using the initial guess $j_{\mu,l+1}-\frac{\pi}{2}$, $i\leq l \leq i+k-1$. Now consider the interval $(0,-(\mu+\frac{1}{2})]$. Clearly $r(x)\geq 1$ in $(0,-(\mu+\frac{1}{2})]$. Using the inequality (see \cite[eq. 2]{c2}, \cite[p. 1313]{boundz}) one can get $j_{\mu,1}>\sqrt{(\mu+1)(\mu+5)}$. Note that for $-0.95\leq \mu<-\frac{1}{2}$; $\sqrt{(\mu+1)(\mu+5)}\geq -(\mu+\frac{1}{2})$. Consequently, the interval $x\in (0, -(\mu+\frac{1}{2})]$ does not contain any positive zeros of $J_\mu(x)$. If $-1<\mu < -0.95$, then the interval contains at most one zero and at most one singularity of $h(x)$ by Theorem 2.1 of \cite{2003}. More specifically, $(\sqrt{(\mu+1)(\mu+5)},-(\mu+\frac{1}{2}))$  contains at most one zero and at most one singularity of $h(x)$. Thus $J_\mu(x)$ has at most one zero in $(\sqrt{(\mu+1)(\mu+5)},-(\mu+\frac{1}{2}))$.  In this situation, one can find $j_{\mu,1}$ as follows. Using the relation $0<j_{\mu, 1}<j_{\mu+1, 1},$ between zeros of $J_{\mu}(x)$ and $J_{\mu+1}(x)$ (see \cite[15.22, p. 479]{Wa22}) one can conclude that $h(x)<0$ in $(\sqrt{(\mu+1)(\mu+5)},j_{\mu,1})$. Thus, all the hypotheses of Theorem \ref{thm2} hold. Hence, $\sqrt{(\mu+1)(\mu+5)}$ is a suitable initial guess for the proposed third-order iterative procedure to find $j_{\mu,1}$. Note that $\Omega(x)=1+\dot{r}(x)-r^2(x)=1+\frac{\frac{1}{4}-\mu^2}{x^2}.$ Hence, for $-\frac{1}{2}\leq\mu<\frac{1}{2}$, $x\mapsto \Omega(x)$ is decreasing on $(0, \infty)$ and for $-1<\mu<-\frac{1}{2}$, $x \mapsto \Omega(x)$ is increasing on $(0, \infty)$. Thus, using Remark \ref{remark5}, \eqref{convex}, and \eqref{concave}, one can obtain a better initial guess, as discussed.

	\subsection{Cylinder function}\label{cylinder} In this subsection, using the proposed third-order iterative method, all the positive zeros of Cylinder function $C_{\mu}(x)$ of order $\mu \in \mathbb{R}$, in a given interval $[a',b']$ are obtained. Cylinder function of order $\mu$ is defined as follows (see \cite{EL98}): $$C_{\mu}(x)=J_{\mu}(x)\cos\alpha-Y_{\mu}(x)\sin\alpha, \hspace{3mm} 0\leq\alpha<\pi,$$ where $Y_{\mu}(x)$ is the Bessel function of second kind. For $k\in \mathbb{N}$, $z_{\mu,k}$ denotes the $k$-th positive zero of $C_{\mu}(x)$. Using the recurrence relation \cite[eq. 10.6.2]{OLB10}, one can verify that the pairs $(C_{\mu}(x),C_{\mu-1}(x))$ and $(C_{\mu}(x),C_{\mu+1}(x))$ satisfies the coupled differential equations \eqref{e1} and \eqref{e2}, respectively. 
	
	\textbf{Case 1 $\mu\geq\frac{1}{2}$:}
	For the choice $y(x)=C_{\mu}(x)$, $w(x)=C_{\mu-1}(x)$ and using the CDE \eqref{e1}, one can verify that all the conclusions of Theorem \ref{thm1} hold true in $(0, \infty)$. Consequently, $h(x)=\frac{C_{\mu}(x)}{C_{\mu-1}(x)}$ satisfies the Riccati equation \eqref{ricatti}, with $r(x)=\frac{\mu-\frac{1}{2}}{x}$. Note that for $\mu=\frac{1}{2}$, $r(x)\equiv 0$. Consequently, $h(x)=\tan (x+\alpha)$. Hence, the positive zeros of $C_{\frac{1}{2}}(x)$ are $n\pi-\alpha$, $n\in\mathbb{N}$. For $x\in (\mu-\frac{1}{2}, \infty)$ and $\mu>\frac{1}{2}$, it is easy to verify that $x\mapsto r(x)$ is decreasing and $0<r(x)<1$.  Let $[a',b'] \subset (\mu-\frac{1}{2},\infty)$ and  contain $k+1$-zeros. i.e., $z_{\mu,l} \in [a',b']$, $i \leq l \leq i+k$ for some $i\in \mathbb{N}$. To find all the zeros in $[a',b']$,  one can get the initial guesses from Table \ref{initial guess on [a', b']}. i.e., the process will start with finding the largest zero in this interval. If $h(b')= 0$, then the largest zero in this interval is $z_{\mu,i+k}=b'$ otherwise the initial guess for $z_{\mu,i+k}$ is $b'-\frac{\pi}{2}\left(\frac{1-sign(h(b'))}{2}\right)$. The rest of the zeros $z_{\mu,l}$ can be obtained using the initial guess $z_{\mu,l+1}-\frac{\pi}{2}$, $i\leq l \leq i+k-1$. Now consider the interval $(0, \mu-\frac{1}{2})$.
	By using \cite[Corollary 1]{c1} with Concluding remark in \cite{c1}, one can get $$z_{\mu, 1}>\mu ~~~~~\text{if}~~~~~ \hspace{1mm} 0\leq\alpha<\theta_1(\mu, \mu)$$ where $\theta_1(\mu, \mu)=\frac{\pi}{2}-\arctan\left(\frac{Y_{\mu}(\mu)}{J_{\mu}(\mu)}\right)$. Note that $\theta_1(\mu, \mu)>\frac{5\pi}{6}$ (see \cite[p. 1232]{c1}). Thus, for $0\leq\alpha<\theta_1(\mu, \mu)$, the interval $(0, \mu-\frac{1}{2})$ does not contain any zero of $C_{\mu}(x)$. Once again using the concluding remark of \cite{c1}, for $\mu\geq0$ and $\theta_1(\mu, \mu)\leq\alpha<\pi$; $C_{\mu}(x)$ has exactly one zero in the interval $(0, \mu]$. The case $z_{\mu, 1}\in (\mu-\frac{1}{2}, \mu]$ is part of $(\mu-\frac{1}{2},\infty)$ just discussed above. 
	
	If $z_{\mu,1}\in (0, \mu-\frac{1}{2}]$, then $z_{\mu, 1}$ can be obtained using the proposed iterative method \eqref{iter} in the following way. Note that the zeros of the cylinder function are increasing with respect to $\mu$ \cite[p. 68]{cylinder}. Consequently, we have $z_{\mu-1, 1}<z_{\mu, 1}$. Moreover, $h(z)<0$ in $(z_{\mu-1,1}, z_{\mu, 1})$. Using part \textit{ (1)} of Theorem \ref{thm2}, one can conclude that the proposed iterative method with any initial guess from $(z_{\mu-1,1}, z_{\mu, 1})$ converges to $z_{\mu, 1}$. Note that $\Omega(x)=1+\dot{r}(x)-r^2(x)=1-\left(\frac{\mu^2-\frac{1}{4}}{x^2}\right)$ and  $x \mapsto \Omega(x)$ is increasing on $(0, \infty)$. Hence, one can get a better initial guess using Remark \ref{remark5} and \eqref{concave}.
	
	

	\textbf{Case 2 $\mu<\frac{1}{2}$:} For the choice $y(x)=C_{\mu}(x)$, $w(x)=C_{\mu+1}(x)$ and using the CDE \eqref{e2}, one can get all conclusions of Theorem \ref{thm1}. Furthermore, $h(x)=\frac{\tilde{y}(x)}{\tilde{w}(x)}=-\frac{C_{\mu}(x)}{C_{\mu+1}(x)}$ satisfies the Riccati differential equation \eqref{ricatti}, with $r(x)=-\frac{\mu+\frac{1}{2}}{x}$. Note that for $\mu=-\frac{1}{2}$, $r(x) \equiv 0$. Consequently, $h(x)=-\cot(x+\alpha)$. Hence, the positive zeros of $C_{-\frac{1}{2}}(x)$ are $(n+\frac{1}{2})\pi-\alpha$, $n\in \mathbb{N}$. We will handle the rest of the situations in two sub-cases.

	\textbf{Sub case 2(a) $\mu<-\frac{1}{2}$:} For $x \in (-(\mu+\frac{1}{2}), \infty)$ and $\mu<-\frac{1}{2}$, it is easy to verify that $x \mapsto r(x)$ is decreasing and $0<r(x)<1$. Let $[a',b'] \subset (-(\mu+\frac{1}{2}),\infty)$ contain $k+1$-zeros. i.e., $z_{\mu,l} \in [a',b']$, $i \leq l \leq i+k$ for some $i\in \mathbb{N}$. To find all zeros in $[a',b']$,  one can get the initial guesses from Table \ref{initial guess on [a', b']}. i.e., the process will start with finding the largest zero in this interval. If $h(b')= 0$, then the largest zero in this interval is $z_{\mu,i+k}=b'$ otherwise the initial guess for $z_{\mu,i+k}$ is $b'-\frac{\pi}{2}\left(\frac{1-sign(h(b'))}{2}\right)$. The rest of the zeros $z_{\mu,l}$ can be obtained using the initial guess $z_{\mu,l+1}-\frac{\pi}{2}$, $i\leq l \leq i+k-1$. Now consider the interval $(0, -(\mu+\frac{1}{2}))$. In view of \cite[Theorem 2.1]{2003}, one can conclude that the interval $(0, -(\mu+\frac{1}{2})]$ contains at most one zero and at most one singularity of $h(x)$.  
	
	If $z_{\mu,1}\in (0, -(\mu+\frac{1}{2})]$, then $z_{\mu, 1}$ can be obtained using the proposed iterative method \eqref{iter} in the following way. Note that zeros of the cylinder function are increasing with respect to $\mu$ \cite[p. 68]{cylinder}. Consequently, we have $z_{\mu, 1}<z_{\mu+1, 1}$. Hence, the interval $(0, z_{\mu, 1})$ does not contain any singularity of $h(z)$.  Now using part \textit{(1)} of Theorem \ref{thm2}, one can conclude that the proposed iterative method with any initial guess from  $(0, z_{\mu, 1})$ converges to $z_{\mu, 1}$.
	
	
	
	\textbf{Sub case 2(b) $-\frac{1}{2}<\mu<\frac{1}{2}$:} For $x\in (0, \infty)$, $r(x)$ is negative and $x \mapsto r(x)$ is increasing. From \cite[p. 490]{Wa22}, all the positive zeros $z_{\mu,k}$ of the cylinder function satisfy
	\begin{equation*}
		(k-1)\pi+\frac{3}{4}\pi+\frac{\mu\pi}{2}-\alpha<z_{\mu, k}<(k-1)\pi+\pi-\alpha, \hspace{2mm}k \in \mathbb{N}.
	\end{equation*} 
	Replacing $k=2$, in the above inequality, one gets $z_{\mu, 2}>\frac{3}{4}\pi+\frac{\mu\pi}{2}+\pi-\alpha>\frac{3}{4}\pi+\frac{\mu\pi}{2}.$ Thus, the interval $\left(\frac{3}{4}\pi+\frac{\mu\pi}{2}, \infty\right)$ contains all positive zeros of $C_{\mu}(x)$ except $z_{\mu, 1}$. For $x\in \left(\frac{3}{4}\pi+\frac{\mu\pi}{2}, \infty\right)$ one can verify that $$r(x)>-\frac{\mu+\frac{1}{2}}{\frac{\pi}{2}(\mu+\frac{3}{2})} ~~~~ \text{and}~~~ \dot{r}(x)<\frac{\mu+\frac{1}{2}}{\frac{\pi^2}{4}(\mu+\frac{3}{2})^2}.$$ For the choice $k_2=\frac{\mu+\frac{1}{2}}{\frac{\pi}{2}(\mu+\frac{3}{2})}$ and $k_1=\frac{\mu+\frac{1}{2}}{\frac{\pi^2}{4}(\mu+\frac{3}{2})^2}$, one can get $8k_2^2+6k_1-3=\frac{4(\mu+\frac{1}{2})}{\pi^2(\mu+\frac{3}{2})^2}\left[8\left(\mu+\frac{1}{2}\right)+6\right]-3.$ For $\mu \in [0, \frac{1}{2})$, one gets $8k_2^2+6k_1-3<\frac{16}{9\pi^2}(8+6)-3<0$. Similarly, for $\mu \in (-\frac{1}{2}, 0)$, one gets $8k_2^2+6k_1-3<\frac{2}{\pi^2}(4+6)-3<0.$
	Let $[a',b'] \subset (\frac{3}{4}\pi+\frac{\mu\pi}{2},\infty)$ contain $k+1$-zeros. i.e., $z_{\mu,l} \in [a',b']$, $i \leq l \leq i+k$ for some $i\in \mathbb{N}$. To find all zeros in $[a',b']$, one can get the initial guesses from Table \ref{initial guess on [a', b']}. i.e., the process will begin by finding the smallest zero within this interval. If $h(a')= 0$, then the smallest zero in this interval is $z_{\mu,i}=a'$ otherwise the initial guess for $z_{\mu,i}$ is $a'+\frac{\pi}{2}\left(\frac{1+sign(h(a'))}{2}\right)$. The rest of the zeros $z_{\mu,l}$ can be obtained using the initial guess $z_{\mu,l-1}+\frac{\pi}{2}$, $i+1\leq l \leq i+k$.

	Note that $\Omega(x)=1+\dot{r}(x)-r^2(x)=1+\left(\frac{-\mu^2+\frac{1}{4}}{x^2}\right).$ It is easy to observe that for $\mu \in \left(-\frac{1}{2}, \frac{1}{2}\right)$, $x \mapsto \Omega(x)$ is decreasing on $(0, \infty)$. For $\mu \leq -\frac{1}{2}$, $x \mapsto \Omega(x)$ is increasing in $(0, \infty)$. Thus, in both the above sub-cases, using Remark \ref{remark5}, \eqref{concave}, and \eqref{convex}, one can get a better initial guess.

    \section{Numerical results} \label{numerical results}

       In this section, the performance of the proposed third-order iterative method \eqref{iter} is demonstrated by various numerical simulations. More specifically, the performance of the proposed third-order method (TOM) is compared with the fourth-order method for Legendre polynomial (FOM-L) \cite{seg11}, second-order method for Legendre polynomial (SOM-L) \cite{2002}, fourth-order method for Hermite polynomial (FOM-H) \cite{seg1}, asymptotic method (ASY) \cite{alex1} for Hermite polynomial, second-order method for Hermite polynomial (SOM-H)\cite{2002}, fourth-order method for Bessel function (FOM-B) \cite{2010}, second-order method for Bessel function (SOM-B) \cite{2002}, modified Newton method for Bessel function (MNM-B) \cite{seg2},  fourth-order method for Cylinder function (FOM-C) \cite{2010}, second-order method for Cylinder function (SOM-C) \cite{2002}, and modified Newton method for Cylinder function (MNM-C) \cite{seg2}. A. Time denotes the average CPU run time (in seconds) obtained by executing the corresponding algorithm ten times and reporting the mean of the measured runtimes. T. Iter denotes the total number of iterations taken by the corresponding algorithm to perform the complete task.  All the codes are available in \cite{code}. All numerical experiments were performed using MATLAB R2024b (64-bit) on a 64-bit Windows PC equipped with a 12th Gen Intel Core i7 processor and 16 GB RAM.

    We follow the approach in \cite{seg1,seg11} to calculate relative error. Let $x_1^{(d)}, x_2^{(d)}, \dots, x_k^{(d)}$ denote the computed zeros of the function $y(x)$ using double precision arithmetic, and let $x_1^{(q)}, x_2^{(q)}, \dots, x_k^{(q)}$ denote the corresponding zeros computed using extended precision arithmetic. The relative error (RE) for each zero is defined as
    \begin{equation}\label{error1}
     RE ~\mbox{at}~ x_i = \left|1 - \frac{x_i^{(d)}}{x_i^{(q)}}\right|,~~~ i=1,2,\cdots k.
    \end{equation}
    To compare the performance of the second-order method \cite{2003,seg2,2002} and the proposed third-order method \ref{iter} with the fourth-order method \cite{seg1,seg11,2010}, the maximum relative error with respect to the fourth-order method MR.Er-F is used. It is defined by 
    \begin{equation} \label{MR.Er-F}
      \mbox{MR.Er-F} =  \max_{ i=1,2,\cdots,k} \mbox{(Er-F)}_i~, ~~\mbox{where}~~ \mbox{(Er-F)}_i=\left|1 - \frac{x_i^{(cd)}}{x_i^{(fd)}}\right|,
    \end{equation}
    $x_i^{(fd)}$ and $x_i^{(cd)}$ denote the $i^{\mbox{th}}$ zero of the function $y(x)$ obtained in double precision by the fourth-order method and the method which is compared with the fourth-order method, respectively. 

     \subsection{Legendre polynomial}
     In this subsection, the proposed method is used to find all the zeros of $P_n(x)$ and compared with methods FOM-L \cite{seg11} and SOM-L \cite{2002}. First, we present the implementation procedure for the proposed third-order iterative method. Using the procedure discussed in Section \ref{Legendre} with initial guess $\frac{\pi}{2}$, one can get all the zeros of $P_n(x)$ by employing the proposed third-order iterative method 
     \begin{equation}\label{leg-it}
     z_{k+1}=z_k-\frac{2h(z_k)}{2+h^2(z_k)+2\tanh(\frac{z_k}{n+1})h(z_k)},
     \end{equation}
     where $h(z)=t(x(z))=-\frac{P_n(x(z))}{P_{n+1}(x(z))}$, $x(z)=\tanh(\frac{z}{n+1})$. Expressing the iterative method \eqref{leg-it} in terms of $x$ one can get
     \begin{align}
     \tanh^{-1}(x_{k+1})&=\tanh^{-1}(x_k)-\frac{2t(x_k)}{(n+1)(2+t^2(x_k)+2x_kt(x_k))} \nonumber\\
     x_{k+1}&=\frac{x_k-\tanh(b(x_k))}{1-x_k\tanh(b(x_k))}, \label{leg_iter}
     \end{align}
     where $ b(x)=\frac{2t(x)}{(n+1)(2+t^2(x)+2xt(x))}.$
     Consequently, during the implementation, one has to evaluate $\frac{P_n(x)}{P_{n+1}(x)}$ at each step of the iterative procedure. The local Taylor series method discussed in \cite{seg11, i1} is used for evaluating the ratio $\frac{P_n(x)}{P_{n+1}(x)}$. Let $y(x)=\lambda\sqrt{1-x^2}P_n(x)$, where $\lambda$ is a normalization constant. Consequently, $\frac{P_n(x)}{P_{n+1}(x)}=\frac{(n+1)y(x)}{nxy(x)-(1-x^2)y'(x)}$. Hence, $t(x)=\frac{(n+1)y(x)}{(1-x^2)y'(x)-nxy(x)}$. To calculate $y(x)$ and $y'(x)$ we use the following truncated Taylor series with suitable center $\delta$
     \begin{equation}\label{taylor-2}
      y(x)=\sum_{k=0}^{N}a_k(\delta)(x-\delta)^k, \hspace{2mm} y'(x)=\sum_{k=0}^{N}(k+1)a_{k+1}(\delta)(x-\delta)^k
     \end{equation}
     where $a_k(\delta)=\frac{y^{(k)}(\delta)}{k!}$. During the numerical simulation, we set $N=\min\{100,N_0\}$ where $N_0$ is the least integer satisfying $\zeta_{N_0}(x)=\max\left(\left|\dfrac{a_{N_0}(\delta)(x-\delta)^{N_0}}{\sum_{k=0}^{N_0}a_k(\delta)(x-\delta)^k}\right|, \left|\dfrac{(N_0+1)a_{N_0+1}(\delta)(x-\delta)^{N_0}}{\sum_{k=0}^{N_0}(k+1)a_{k+1}(\delta)(x-\delta)^k}\right|\right)<10^{-19}.$ Let $Q(x)=4(1-x^2)^2$, and $R(x)=(4n^2+4n)(1-x^2)+4$. Note that $y$ satisfies the following equations (see \cite[eq. 26]{seg11})
     \begin{eqnarray}
      \label{L-Ev-D1}
     &&2Q(x)a_2(x)+R(x)a_0(x)=0\\
     &&6Q(x)a_3(x)+2Q'(x)a_2(x)+R(x)a_1(x)+R'(x)a_0(x)=0\\
     &&(k+2)(k+1)Q(x)a_{k+2}(x)+(k+1)kQ'(x)a_{k+1}(x)+\left(\frac{k(k-1)}{2}Q''(x)+R(x)\right)a_k(x) \nonumber\\
     \label{L-Ev-RR}
    &&+\left(\frac{(k-1)(k-2)}{6}Q'''(x)+R'(x)\right)a_{k-1}(x) \\
    &&+\frac{1}{2}\left(\frac{(k-2)(k-3)}{12}Q^{(4)}(x)+R''(x)\right)a_{k-2}(x)=0, \hspace{5mm} k=2,3,\cdots .\nonumber
    \end{eqnarray}
    Thus if we know $a_0(\delta)$ and $a_1(\delta)$, using \eqref{L-Ev-D1}--\eqref{L-Ev-RR}, one can easily find $a_k(\delta)$ for all $k\geq 2$. Now, choose $\lambda=\frac{1}{P_n'(0)}$ if $n$ is odd. Subsequently, $y(0)=0$ and $y'(0)=1$. Similarly, if $n$ is even, choose $\lambda=\frac{1}{P_n(0)}$. Subsequently, $y(0)=1$ and $y'(0)=0$. 

Let $x_0^1$ be the initial guess for the first positive zero $x^1$ of $P_n(x)$. To get the first iterate $x_1^1$, $y(x_0^1)$ and $y'(x_0^1)$ can be calculated using eq. \eqref{taylor-2} with center $\delta=0$. Let $m \in \mathbb{N}$. Note that, at the end of $m^{\mbox{th}}$ iteration the following values $x_m^1$, $y(x_{m-1}^1)$ and $y'(x_{m-1}')$ are readily available. To get the $m+1^{\mbox{th}}$ iterate $x_{m+1}'$, $y(x_m^1)$ and $y'(x_m^1)$ can be calculated using eq. \eqref{taylor-2} with center $\delta=x_{m-1}^1$. The following stopping criteria $\left|\frac{x_{m+1}^1-x_m^1}{x_m^1}\right|\leq 10^{-15}$ is used to stop this iterative procedure. To compute the second positive zero $x^2$, the initial guess is taken as $x_0^2=x_{m+1}^1+\tanh(\frac{\pi}{2(n+1)})$. To get the first iterate $x_1^2$, $y(x_0^2)$ and $y'(x_0^2)$ can be calculated using eq. \eqref{taylor-2} with center $\delta=x_m^1$ and repeat the above procedure. Continuing this process, one can get all remaining positive zeros of $y(x)$. Furthermore, as mentioned in Remark \ref{remark5}, one can use improved initial approximations for all but the first three zeros.  

Now, the proposed method TOM is compared with FOM-L \cite{seg11} and SOM-L \cite{2002} for finding zeros of the Legendre polynomials. The MATLAB codes used for this comparison study are based on the code available on the website \footnote{https://personales.unican.es/segurajj/gaussian.html}. The modified version of this code for finding zeros of the Legendre polynomials is available in \cite{code}. It is worth mentioning that the implementation procedures of TOM, FOM-L, and SOM-L depend on the evaluation of ratios involving suitable Legendre polynomials. All these ratios are evaluated using the local Taylor series method discussed in \cite{seg11, i1}.

Figure \ref{RE_P} presents the relative error \eqref{error1} at the zeros of $P_{10000}(x)$ computed using FOM-L, SOM-L, and TOM. From the figure, it is clear that the proposed method TOM provides good accuracy.
 
    Table~\ref{table_comparison_legendre} presents the total number of iterations and the average time taken by FOM-L \cite{seg11}, SOM-L \cite{2002}, and the proposed method TOM \eqref{iter} for computing all the zeros of selected Legendre polynomials of various degrees in the range $[10^{6}, 13\times10^{5}]$. For this problem, the iterative methods TOM and FOM-L \cite{seg11} take the following form 
    \begin{equation}
        x_{k+1}=\frac{x_k-\tanh(\phi(x_k))}{1-x_k\tanh(\phi(x_k))}
    \end{equation}
    with the suitable choice of $\phi(x)$. More specifically, for the choice $\phi(x)=\frac{2t(x)}{(n+1)(2+t^2(x)+2xt(x))}$, $t(x)=\frac{(n+1)y(x)}{(1-x^2)y'(x)-nxy(x)}$ and $\phi(x)=\frac{1}{\sqrt{\mathcal{A}(x)}}\arctan(\sqrt{\mathcal{A}(x)}t_1(x))$, $\mathcal{A}(x)=n(n+1)(1-x^2)$, $t_1(x)=\frac{y(x)}{(1-x^2)y'(x)+xy(x)}$, one can get TOM and FOM-L, respectively. 
    In Table \ref{table_comparison_legendre}, the performance of the proposed method and the second-order method \cite{2002} is compared with the fourth-order method \cite{seg11}. As expected, FOM-L and SOM-L take the minimum and maximum number of iterations, respectively, to find all the zeros of various $P_n(x)$. It is interesting to note that the MR.Er-F for the proposed TOM-L is in the order of $10^{-11}$ due to a few zeros close to one. For example, when $n=1300000$, only thirteen zeros close to one have $\mbox{(Er-F)}_i$ greater than the order of $10^{-16}$; the rest of the zeros have $\mbox{(Er-F)}_i$ in the order of $10^{-16}$. Similarly, the MR.Er-F for the SOM-L is in the order of $10^{-10}$ due to a few zeros close to one. For example, when $n=1300000$, only twenty two zeros close to one have $\mbox{(Er-F)}_i$ greater than the order of $10^{-16}$; the rest of the zeros have $\mbox{(Er-F)}_i$ in the order of $10^{-16}$. Hence, compared to the SOM-L \cite{2002}, the zeros produced by the proposed third-order method are relatively close to the zeros produced by FOM-L \cite{seg11}.

    Figure \ref{CPU_P_n} shows the average of the CPU times for the methods FOM-L \cite{seg11}, SOM-L \cite{2002}, and the proposed method TOM for finding all the zeros of $P_n(x)$ for $n \in (10^6, 13\times10^{5})$. More specifically, all the zeros of fifteen polynomials of degree varying from $10^{6}$ to $13\times10^{5}$ obtained using FOM-L, SOM-L and TOM. For each polynomial, every method was executed ten times, and the average CPU time was recorded. From Figure \ref{CPU_P_n}, it is clear that the proposed method TOM is faster than SOM-L\cite{2002} for these polynomials.

\begin{figure}[htbp]
\centering
	\begin{subfigure}{0.48\textwidth}
    \centering
    \includegraphics[width=\linewidth]{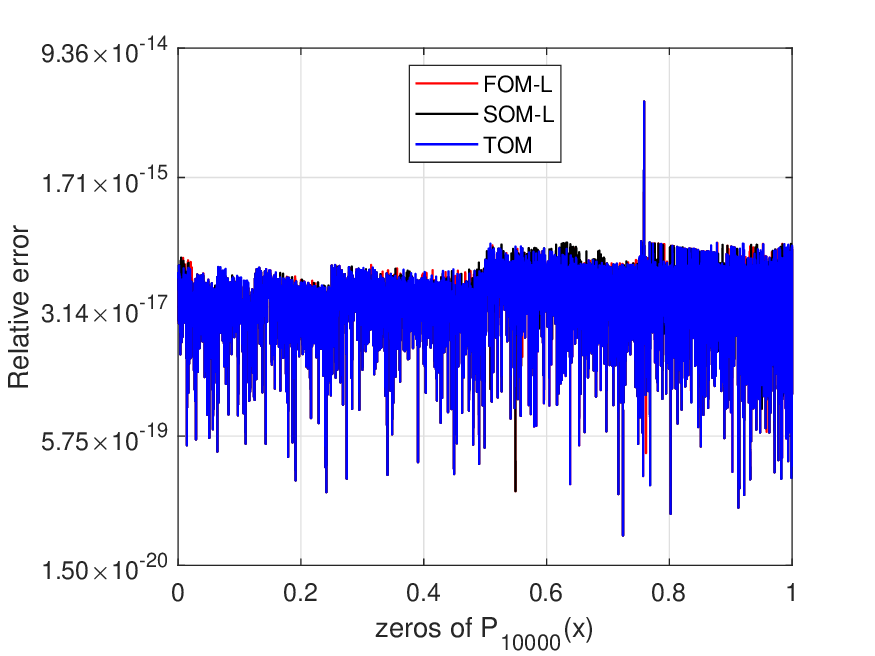}
    \caption{Relative error \eqref{error1} at zeros of $P_{10000}(x)$}
    \label{RE_P}
\end{subfigure}\hfill
\begin{subfigure}{0.48\textwidth}
    \centering
    \includegraphics[width=\linewidth]{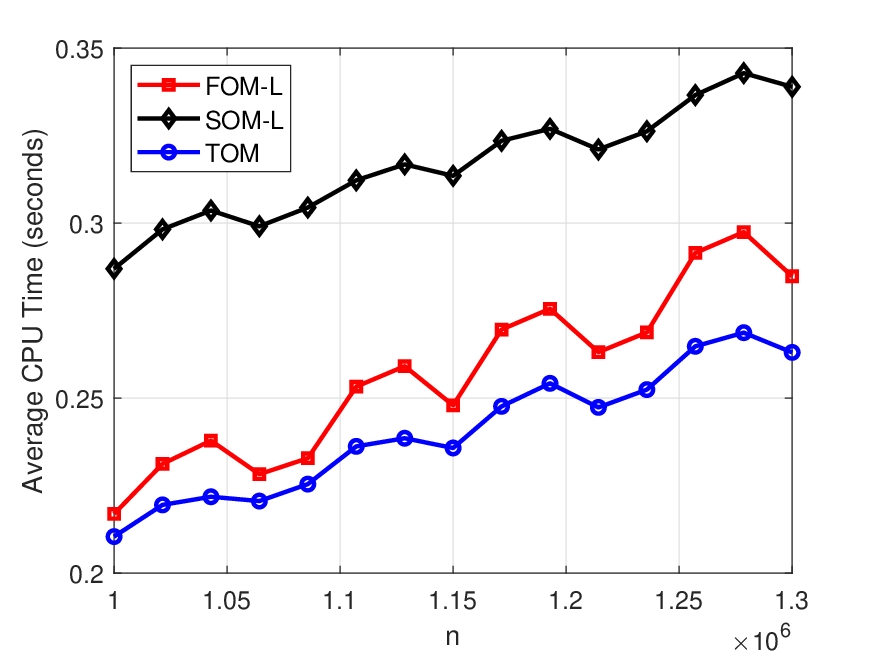}
    \caption{Average CPU time}
    \label{CPU_P_n}
\end{subfigure}

\caption{FOM-L \cite{seg11} vs SOM-L \cite{2002} vs TOM \eqref{iter} in finding all zeros of $P_n(x)$.}
\end{figure}

     \begin{table}[h]
\centering
\small
\setlength{\tabcolsep}{4pt}
\renewcommand{\arraystretch}{1.1}

\begin{tabular}{|c|cc|ccc|ccc|}
\hline
\multirow{2}{*}{\textbf{$n$}}
& \multicolumn{2}{c|}{\textbf{FOM-L} \cite{seg11}}
& \multicolumn{3}{c|}{\textbf{SOM-L} \cite{2002}}
& \multicolumn{3}{c|}{\textbf{Proposed-TOM} \eqref{iter}} \\ \cline{2-9}

& \textbf{T. Iter} & \textbf{A. Time}
& \textbf{MR.Er-F} & \textbf{T. Iter} & \textbf{A. Time}
& \textbf{MR.Er-F} & \textbf{T. Iter} & \textbf{A. Time}   \\ \hline

1000000 & 1000004 & 0.2110 & $1.49\times10^{-10}$ & 1373694 & 0.2810 & $5.96\times10^{-11}$ & 1130692 & 0.2053 \\
1050000 & 1050004 & 0.2157 & $1.55\times10^{-10}$ & 1419737 & 0.2850 & $6.82\times10^{-11}$ & 1173888 & 0.2108 \\
1100000 & 1100004 & 0.2264 & $1.51\times10^{-10}$ & 1470889 & 0.2947 & $6.22\times10^{-11}$ & 1229188 & 0.2208 \\
1200000 & 1200004 & 0.2484 & $1.55\times10^{-10}$ & 1560124 & 0.3121 & $7.08\times10^{-11}$ & 1321758 & 0.2402 \\
1300000 & 1300003 & 0.2734 & $1.65\times10^{-10}$ & 1650345 & 0.3260 & $7.64\times10^{-11}$ & 1417543 & 0.2540 \\ \hline
\end{tabular}

\caption{Computations of  all zeros of $P_n(x)$ for different values of $n$: FOM-L \cite{seg11}, SOM-L \cite{2002}, and TOM \eqref{iter}.}
\label{table_comparison_legendre}
\end{table}




\subsection{Hermite polynomial}	
      In this subsection, the proposed method is used to find all the zeros of $H_n(x)$ and compared with  ASY \cite{alex1}, FOM-H \cite{seg1}, and SOM-H \cite{2002}. First, we present the implementation procedure for the proposed third-order iterative method. Using the procedure discussed in Section \ref{HP-T}, with initial guess $\frac{\pi}{2}$, one can get all the zeros of $H_n(x)$ by employing the proposed third-order iterative method 
      \begin{equation}\label{it-He-toh-z}
          z_{k+1}=z_k-\frac{2h(z_k)}{2+h^2(z_k)+\frac{z_k}{n+1}h(z_k)},
      \end{equation}
      where $h(z)=t(x(z))=-\sqrt{2(n+1)}\frac{H_n(x(z))}{H_{n+1}(x(z))}$ and $x(z)=\frac{z}{\sqrt{2(n+1)}}$. Expressing the iterative method \eqref{it-He-toh-z} in terms of $x$ one can get
      \begin{equation}\label{it-He-toh-x}
          x_{k+1}= x_k-\frac{2 t(x_k)}{\sqrt{2(n+1)}(2+t^2(x_k))+2x_kt(x_k)}.
      \end{equation}
      Consequently, during the implementation, one has to evaluate $\frac{H_n(x)}{H_{n+1}(x)}$ at each step of the iterative procedure. The local Taylor series method discussed in \cite{seg1,i1} is used for evaluating the ratio $\frac{H_n(x)}{H_{n+1}(x)}$. Let $y(x)=\lambda{\rm e}^{-\frac{x^2}{2}}H_n(x)$, where $\lambda$ is a normalization constant. Consequently, $\frac{H_n(x)}{H_{n+1}(x)}=\frac{y(x)}{xy(x)-y'(x)}$. Hence, $t(x)=-\frac{\sqrt{2(n+1)}y(x)}{xy(x)-y'(x)}$. To calculate $y(x)$ and $y'(x)$ we use the following truncated Taylor series with suitable center $\delta$
      \begin{equation}\label{Taylor1}
      y(x) = \sum_{k=0}^{N} \frac{y^{(k)}(\delta)}{k!} (x-\delta)^k, \quad 
	   y'(x) = \sum_{k=0}^{N} \frac{y^{(k+1)}(\delta)}{k!} (x-\delta)^k.
       \end{equation}
       During the numerical simulation,  we set $N=\min\{50,N_0\}$ where $N_0$ is the least integer satisfying
      $\zeta_{N_0}(x) = \left|\frac{\frac{y^{(N_0+1)}(\delta) (x-\delta)^{N_0}}{N_0!}}{\sum_{k=0}^{N_0} \frac{y^{(k+1)}(\delta) (x-\delta)^k}{k!}}\right| < 10^{-25}.$ Note that $y$ satisfies the following differential equations (see \cite[eq. 22]{seg1})
      \begin{eqnarray} 
      \label{N-Ev-D1}
          && y''(x)+(2n+1-x^2)y(x)=0.\\
          && y'''(x)+(2n+1-x^2)y'(x)-2xy(x)=0.\\
          \label{N-Ev-RR}
          && y^{(k+2)}(x) + (2n+1 - x^2) y^{(k)}(x) - 2k x y^{(k-1)}(x) - k(k-1) y^{(k-2)}(x) = 0, \mbox{for $k=2,3,\cdots$}.
      \end{eqnarray}
      Thus, if we know $y(\delta)$ and $y'(\delta)$, using \eqref{N-Ev-D1} -- \eqref{N-Ev-RR}, one can easily find $y^{(k)}(\delta)$ for all $k\geq 2$. Now, choose $\lambda=\frac{1}{H_n'(0)}$ if $n$ is odd. Subsequently, $y(0)=0$ and $y'(0)=1$. Similarly, if $n$ is even, choose $\lambda=\frac{1}{H_n(0)}$. Subsequently, $y(0)=1$ and $y'(0)=0$.

      Let $x^1_0$ be the initial guess for the first positive zero $x^1$ of $H_n(x)$. To get the first iterate $x^1_1$, $y(x^1_0)$ and $y'(x^1_0)$ can be calculated using Eqn. \eqref{Taylor1} with center $\delta=0$. Let $m \in \mathbb{N}$. Note that, at the end of $m^{\mbox{th}}$ iteration the following values $x^1_m$, $y(x^1_{m-1})$ and $y'(x^1_{m-1})$ are readily available. To get the $(m+1)^{\mbox{th}}$ iterate $x^1_{m+1}$,  $y(x^1_m)$ and $y'(x^1_m)$ can be calculated using Eqn. \eqref{Taylor1} with center $\delta=x^1_{m-1}$.  The following stopping criteria $\left|\frac{x^1_{m+1}-x^1_m}{x^1_m}\right| < 10^{-10}$ is used to stop this iterative procedure. To compute the second positive zero $x^2$, the initial guess is taken as $x^2_0= x^1_{m+1} + \frac{\pi}{2\sqrt{2(n+1)}}$. To get the first iterate $x^2_1$, $y(x^2_0)$ and $y'(x^2_0)$ can be calculated using Eqn. \eqref{Taylor1} with center $\delta=x^1_{m}$ and repeat the above procedure. Continuing this process, one can get all remaining positive zeros of $y(x)$. Furthermore, as mentioned in Remark \ref{remark5}, one can use improved initial approximations for all but the first three zeros.


    Now, the proposed method TOM is compared with ASY \cite{alex1}, FOM-H \cite{seg1}, and SOM-H \cite{2002}. Note that ASY is an asymptotic-based method. The implementation procedure for ASY is available for the Julia programming language in \cite{alex_code}. For this comparison study, the MATLAB version of \cite{alex_code} is used without changing any parameter, and this modified version is available in \cite{code}. Similarly, the implementation procedure for FOM-H \cite{seg1} is available\footnote{https://personales.unican.es/segurajj/gaussian.html} for the FORTRAN programming language. For this comparison study, the MATLAB version of the code is used, and this modified version is available in \cite{code}. It is worth mentioning that the implementation procedure of  TOM, FOM-H\cite{seg1}, and SOM-H\cite{2002} depends on the evaluation of ratios involving suitable Hermite polynomials. All these ratios are evaluated using the local Taylor series method discussed in \cite{seg1,i1}.  

    Figure \ref{RE} presents the relative error \eqref{error1} at the zeros of $H_{10000}(x)$ computed using FOM-H, SOM-H, ASY, and TOM. From the figure, it is clear that the proposed method TOM provides good accuracy.
\begin{figure}[htbp]
\centering

\begin{subfigure}[t]{0.48\textwidth}
    \centering
    \includegraphics[width=\linewidth]{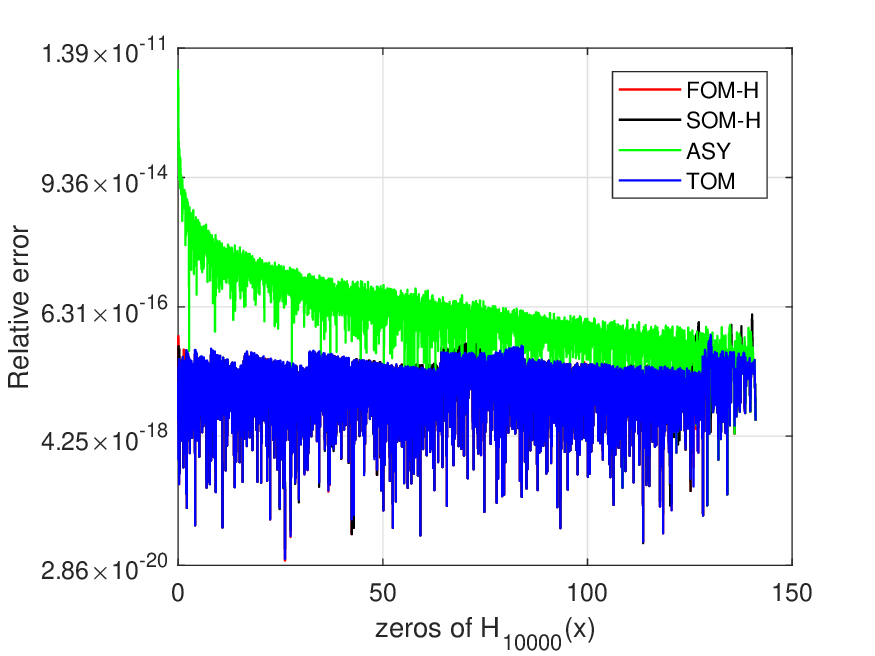}
    \caption{Relative error \eqref{error1} at zeros of $H_{10000}(x)$}
    \label{Hermite_profile}
\end{subfigure}
\hfill
\begin{subfigure}[t]{0.48\textwidth}
    \centering
    \includegraphics[width=\linewidth]{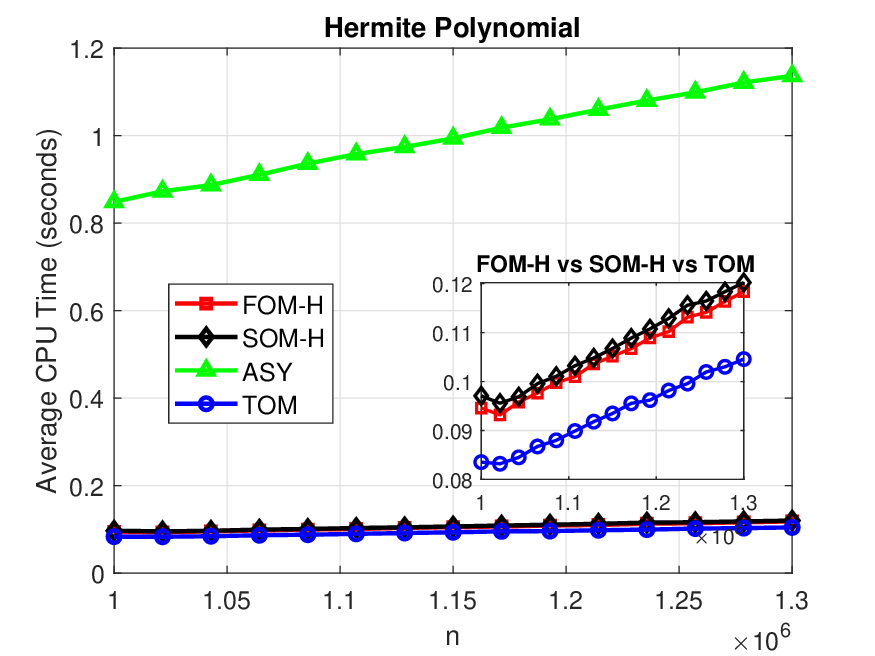}
    \caption{Average CPU time}
    \label{CPU_hermite}
    \label{RE}
\end{subfigure}

\caption{FOM-H \cite{seg1} vs SOM-H \cite{2002} vs ASY \cite{alex1} vs TOM \eqref{iter} in finding all zeros of $P_n(x)$.}
\label{fig:error_comparison}
\end{figure}

   Figure~\ref{CPU_hermite} shows the average of ten CPU times for the methods ASY \cite{alex1}, FOM-H \cite{seg1}, SOM-H \cite{2002}, and the proposed method TOM for finding all the zeros of $H_n(x)$ for $n \in (10^6, 13\times10^5)$.  More specifically, all the zeros of fifteen polynomials of degree varying from $10^6$ to $13\times 10^5$ obtained using ASY, FOM-H, SOM-H and the proposed method TOM. For each polynomial, every method was executed ten times, and the average CPU time was recorded. From Figure~\ref{CPU_hermite}, it is clear that the proposed method TOM is faster than the other methods for these polynomials. 

        	



    \begin{table}[h]
\centering
\small
\setlength{\tabcolsep}{4pt}
\renewcommand{\arraystretch}{1.1}

\begin{tabular}{|c|cc|ccc|ccc|}
\hline
\multirow{2}{*}{\textbf{$n$}}
& \multicolumn{2}{c|}{\textbf{FOM-H} \cite{seg1}}
& \multicolumn{3}{c|}{\textbf{SOM-H} \cite{2002}}
& \multicolumn{3}{c|}{\textbf{Proposed-TOM} \eqref{iter}} \\ \cline{2-9}

& \textbf{T. Iter} & \textbf{A. Time}
& \textbf{MR.Er-F} & \textbf{T. Iter} & \textbf{A. Time}
& \textbf{MR.Er-F} & \textbf{T. Iter} & \textbf{A. Time} \\ \hline

1000000 & 508135 & 0.0954 & $1.96\times10^{-14}$ & 525342 & 0.0975 & $2.22\times10^{-16}$ & 524354 & 0.0819 \\
1050000 & 532925 & 0.0967 & $2.09\times10^{-14}$ & 549676 & 0.0992 & $2.22\times10^{-16}$ & 548692 & 0.0841 \\
1100000 & 557731 & 0.1010 & $2.14\times10^{-14}$ & 574063 & 0.1024 & $2.22\times10^{-16}$ & 573080 & 0.0871 \\
1200000 & 607382 & 0.1093 & $2.33\times10^{-14}$ & 622971 & 0.1121 & $2.22\times10^{-16}$ & 621993 & 0.0964 \\
1300000 & 657076 & 0.1189 & $2.54\times10^{-14}$ & 672030 & 0.1209 & $2.22\times10^{-16}$ & 671049 & 0.1040 \\ \hline
\end{tabular}

\caption{Computations of all zeros of $H_n(x)$ for different values of $n$: FOM-H \cite{seg1}, SOM-H \cite{2002}, and TOM \eqref{iter}.}
\label{table_comparison_hermite}
\end{table}






Table~\ref{table_comparison_hermite} presents the total number of iteration and the average time taken by FOM-H \cite{seg1}, SOM-H \cite{2002}, and the proposed method TOM \eqref{iter} for computing all the zeros of selected Hermite polynomials of various degrees in the range $[10^{6}, 13\times 10^{5}]$. For this problem, the iterative methods TOM and FOM-H \cite{seg1} can be written as Eq. \ref{it-He-toh-x}
and 
\begin{equation*}
    x_{k+1} = x_k - \frac{1}{\sqrt{\mathcal{A}(x_k)}}\arctan\left(\sqrt{\mathcal{A}(x_k)}t_1(x_k)\right), \quad t_1(x) = \frac{y(x)}{y'(x)},\quad \mathcal{A}(x) = 2n+1-x^2,
\end{equation*}
respectively. In Table \ref{table_comparison_hermite}, the performance of the proposed method and the second-order method\cite{2002} is compared with the fourth-order method \cite{seg1}. As expected, FOM-H \cite{seg1} and SOM-H \cite{2002} take the minimum and maximum number of iterations, respectively, to find all the zeros of various $H_n(x)$. The maximum relative error MR.Er-F with respect to FOM-H \cite{seg1}, for the proposed third-order method is in the order of $10^{-16}$. However, the maximum relative error MR.Er-F with respect to FOM-H \cite{seg1}, for SOM-H \cite{2002} is in the order of $10^{-14}$. Hence, compared to the SOM-H \cite{2002}, the zeros produced by the proposed third-order method are relatively close to the zeros produced by FOM-H \cite{seg1}. From Table \ref{table_comparison_hermite}, it can be seen that, although the proposed third-order method TOM requires more iterations than the fourth-order method FOM-H \cite{seg1}, TOM takes less time to do the same job than FOM-H. This may be due to the computational complexity involved in repeatedly evaluating functions in the respective iterative procedures.

	\subsection{Bessel function}
    In this subsection, the proposed method is used to find all zeros of $J_{\mu}(x)$ in a given interval and compared with methods MNM-B \cite{seg2} and SOM-B \cite{2002}.  For $\mu>\frac{1}{2}$, all positive zeros of $J_{\mu}(x)$ lie in the interval $(\mu, \infty)$. Using the procedure discussed in Section \ref{Bessel_im}, one can get all zeros of $J_{\mu}(x)$ in the given interval $[a, b]\subseteq [\mu,\infty)$ by employing the proposed third-order iterative method
    \begin{equation}\label{iter_bessel}
        x_{k+1}=x_k-\frac{2h(x_k)}{2+h^2(x_k)-\frac{2\mu-1}{x_k}h(x_k)},
    \end{equation}
    where $h(x)=\frac{J_{\mu}(x)}{J_{\mu-1}(x)}$. The procedure starts with the initial guess $b-\frac{\pi}{2}\left(\frac{1-sign(h(b))}{2}\right)$. Let $x_m^k$ denotes the $m^{\mbox{th}}$ iterate of \eqref{iter_bessel} to obtain the $k^{\mbox{th}}$ zero in the interval $[a, b]$. The following stopping criteria $\left|x_{m+1}^k-x_m^k\right|<10^{-10}$ is used to stop the iteration. Now, the proposed method TOM is compared with the iterative methods MNM-B \cite{seg2}, and SOM-B \cite{2002}. For all three methods, the function $h(x)$ involving the ratios of Bessel functions is evaluated using MATLAB built-in command \texttt{`besselj'}. \cite[Lemma 2.1]{1998} and \cite[Lemma 4.2]{2002} ensure that $b-\frac{\pi}{2}\left(\frac{1-sign(h(b))}{2}\right)$ is a suitable initial guess for MNM-B and SOM-B, respectively.
  	
    Figure \ref{bessel_fig_1} presents the relative error \eqref{error1} at the zeros of $J_{10000}(x)$ in the interval $I_0=[\mu, ~\mu+10000]$ computed using  MNM-B \cite{seg2}, SOM-B \cite{2002}, and TOM. From the figure, it is clear that the proposed method TOM provides good accuracy.

   Figure~\ref{CPU_bessel}, presents the average of ten CPU times for the methods MNM-B \cite{seg2}, SOM-B \cite{2002}, and the proposed method TOM for finding all the zeros of $J_{\mu}(x)$ in the interval $I_1 = [\mu, ~\mu+100000]$ for $\mu \in [10000, 11000]$. More specifically, all the zeros of $J_{\mu}(x)$ in the interval $I_1$ for fifteen different $\mu$ varying from $10000$ to $11000$ are obtained using MNM-B, SOM-B, and the proposed method TOM. For each $\mu$, every method was executed ten times, and the average CPU time was recorded. From Figure \ref{CPU_bessel}, it is clear that the proposed method, TOM, is faster than the other methods.

\begin{figure}[htbp]
\centering

\begin{subfigure}[t]{0.48\textwidth}
    \centering
    \includegraphics[width=\linewidth]{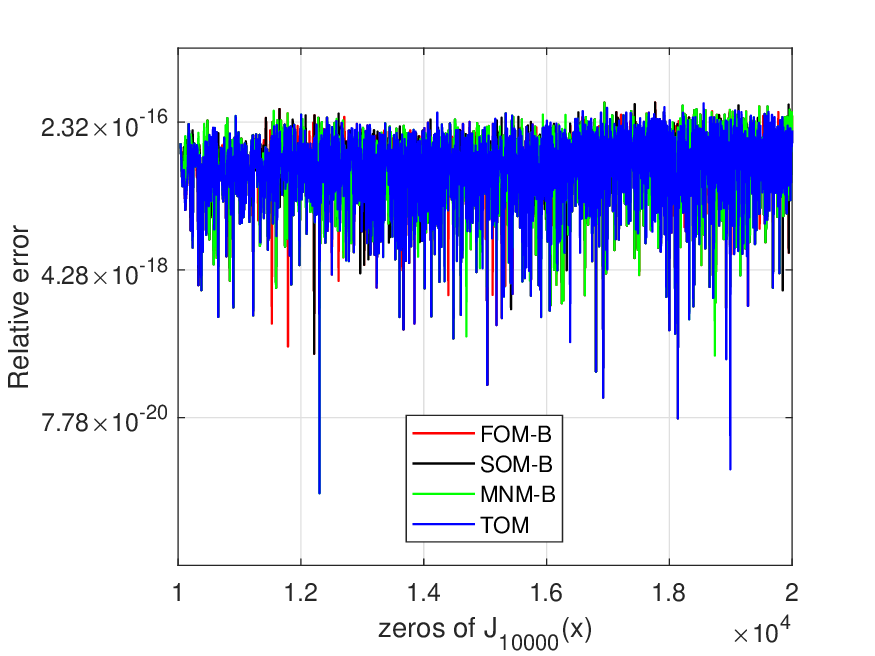}
    \caption{Relative error \eqref{error1} at zeros of $J_{10000}(x)$ in the interval $I_0$}
    \label{bessel_fig_1}
\end{subfigure}
\hfill
\begin{subfigure}[t]{0.48\textwidth}
    \centering
    \includegraphics[width=\linewidth]{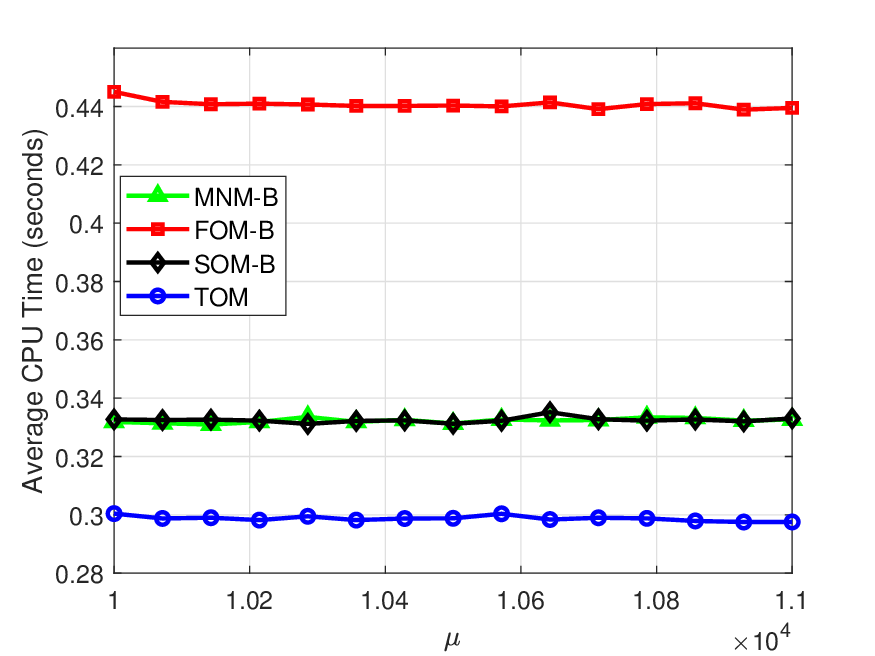}
    \caption{Average CPU runtime for computing all zeros of $J_{\mu}(x)$ in the interval $I_1$ for various choices of $\mu\in [10000, 11000]$}
    \label{CPU_bessel}
\end{subfigure}

\caption{FOM-B \cite{2010} vs MNM-B \cite{seg2} vs SOM-B \cite{2002} vs TOM \eqref{iter} in finding zeros of $J_{\mu}(x)$.}
\end{figure}














\begin{table}[h]
\centering
\small
\setlength{\tabcolsep}{0.8 pt}  
\renewcommand{\arraystretch}{1}

\begin{tabular}{|c|cc|ccc|ccc|ccc|}
\hline
\multirow{2}{*}{\textbf{$\mu$}} 
& \multicolumn{2}{c|}{\textbf{FOM-B} \cite{2010}} 
& \multicolumn{3}{c|}{\textbf{MNM-B} \cite{seg2}} 
& \multicolumn{3}{c|}{\textbf{SOM-B} \cite{2002}} 
& \multicolumn{3}{c|}{\textbf{Proposed-TOM} \eqref{iter}} \\ \cline{2-12}

& \textbf{T.Iter} & \textbf{A.Time} 
& \textbf{MR.Er-F} & \textbf{T.Iter} & \textbf{A.Time}
& \textbf{MR.Er-F} & \textbf{T.Iter} & \textbf{A.Time}
& \textbf{MR.Er-F} & \textbf{T.Iter} & \textbf{A.Time}\\ \hline

1000  & 63326 & 0.4648 & $8.85\times10^{-16}$& 65337 & 0.3159  & $8.85\times10^{-16}$ & 65315 & 0.3160  & $8.85\times10^{-16}$ & 63726 & 0.3114  \\
3000  & 62628 & 0.4576 & $9.41\times10^{-16}$ & 66612  & 0.3225 & $9.41\times10^{-16}$ & 66587 & 0.3235  & $8.91\times10^{-16}$ & 63293 & 0.3087  \\
6000  & 61621 & 0.4508 & $8.38\times10^{-16}$ & 67658 & 0.3298  & $8.38\times10^{-16}$ & 67633 & 0.3279 & $8.38\times10^{-16}$ & 62518 & 0.3043 \\
9000  & 60661 & 0.4442 & $8.26\times10^{-16}$ & 68367 & 0.3310  & $8.26\times10^{-16}$ & 68345 & 0.3310  & $8.26\times10^{-16}$ & 61720 & 0.3003  \\
11000 & 60045 & 0.4381 & $8.12\times10^{-16}$ & 68697 & 0.3321  & $8.12\times10^{-16}$ & 68672 & 0.3313   & $8.12\times10^{-16}$ & 61186 & 0.2976   \\ \hline
\end{tabular}

\caption{Computations of all zeros of $J_{\mu}(x)$ for different values of $\mu$ in the interval $I_1$: FOM-B \cite{2010}, MNM-B \cite{seg2}, SOM-B \cite{2002}, and TOM \eqref{iter}.}
\label{table_comparison_bessel}
\end{table}






Table \ref{table_comparison_bessel} presents the total number of iteration and the average time taken by FOM-B \cite{2010}, MNM-B \cite{seg2}, SOM-B \cite{2002}, and TOM \eqref{iter} for computing all the zeros of selected Bessel functions in the interval $I_1 = (\mu, \mu+100000)$ for various $\mu$ ranges in the interval $[1000, 11000]$. For this problem, the iterative methods TOM and FOM-B can be written as \ref{iter_bessel}
and 
\begin{equation*}
x_{k+1} = x_k - \frac{1}{\sqrt{\mathcal{A}(x_k)}} \arctan\left(\sqrt{\mathcal{A}(x_k)}t_1(x_k)\right), \quad t_1(x) = \frac{J_{\mu}(x)}{J_{\mu-1}+\frac{1}{x}(\frac{1}{2}-\mu)J_{\mu}(x)}, \quad \mathcal{A}(x) = 1+\frac{\frac{1}{4}-\mu^2}{x^2},
\end{equation*}
respectively. In Table \ref{table_comparison_bessel}, the performance of the proposed method and the second-order method\cite{seg2,2002} is compared with the fourth-order method \cite{2010}. As expected, FOM-B \cite{2010} and MNM-B \cite{seg2} take the minimum and maximum number of iterations, respectively, to find all the zeros of various $J_\mu(x)$ in $I_1$. It is interesting to note that the proposed third-order method, MNM-B \cite{seg2}, and SOM-B \cite{2002} have the same maximum relative error MR.Er-F with respect to FOM-B \cite{2010}. For all the methods the MR.Er-F is in the order of $10^{-16}$. From Table \ref{table_comparison_bessel}, it can be seen that, although the proposed third-order method TOM requires more iterations than the fourth-order method FOM-B \cite{2010}, TOM takes less time to do the same job than FOM-B. This may be due to the computational complexity involved in repeatedly evaluating functions in the respective iterative procedures.

\subsection{Cylinder function}

In this subsection, the proposed method is used to find all zeros of $C_{\mu}(x)$ in a given interval and compared with methods  MNM-C \cite{seg2} and SOM-C \cite{2002}. For $\mu>\frac{1}{2}$ and $\alpha<\frac{5\pi}{6}$, all positive zeros of $C_{\mu}(x)$ lie in the interval $(\mu, \infty)$. Using the procedure discussed in Section \ref{cylinder}, one can get all zeros of $C_{\mu}(x)$ in the given interval $[a, b] \subset [\mu,\infty)$ by employing the proposed third-order iterative method
\begin{equation}\label{iter_cylinder}
     x_{k+1}=x_k-\frac{2h(x_k)}{2+h^2(x_k)-\frac{2\mu-1}{x_k}h(x_k)},
\end{equation}
where $h(x)=\frac{C_{\mu}(x)}{C_{\mu-1}(x)}$. The procedure starts with the initial guess $b-\frac{\pi}{2}\left(\frac{1-sign(h(b))}{2}\right)$. Let $x_m^k$ denots the $m^{\mbox{th}}$ iterate of \eqref{iter_cylinder} to obtain the $k^{\mbox{th}}$ zero in the interval $[a, b]$. The following stopping criteria $\left|x_{m+1}^k-x_m^k\right|<10^{-10}$ is used to stop the iteration. Now, the proposed method TOM is compared with the iterative methods MNM-C \cite{seg2}, and SOM-C \cite{2002}. For all three methods, the function $h(x)$ involving the ratios of Cylinder functions is evaluated using MATLAB built-in commands \texttt{`besselj'} and \texttt{`bessely'}. \cite[Lemma 2.1]{1998} and \cite[Lemma 4.2]{2002} ensure that $b-\frac{\pi}{2}\left(\frac{1-sign(h(b))}{2}\right)$ is a suitable initial guess for MNM-C and SOM-C.

Figure \ref{re_2} presents the relative error \eqref{error1} at the zeros of $C_{10000}(x)$ in the interval $I_0 = [\mu, ~\mu+10000]$ and $\alpha=0.75$ computed using MNM-C \cite{seg2}, SOM-C \cite{2002}, and TOM. From the figure, it is clear that the proposed method TOM provides good accuracy.




Figure \ref{CPU_cyl_a}, presents the average of ten CPU times for the methods MNM-C \cite{seg2}, SOM-C \cite{2002}, and the proposed method TOM for finding all the zeros of $C_{\mu}(x)$ in the interval $I_1 = [\mu, ~\mu+100000]$ for $\mu \in [10000, 11000]$ and $\alpha=0.75$. More specifically, all the zeros of $C_{\mu}(x)$ in the interval $I_1$ for fifteen different $\mu$ varying from $10000$ to $11000$ are obtained using  MNM-C, SOM-C and the proposed method TOM. For each $\mu$, every method was executed ten times, and the average CPU time was recorded. From Figure \ref{CPU_cyl_a}, it is clear that the proposed method, TOM, is faster than the other methods.
    
\begin{figure}[htbp]
\centering

\begin{subfigure}[t]{0.48\textwidth}
    \centering
    \includegraphics[width=\linewidth]{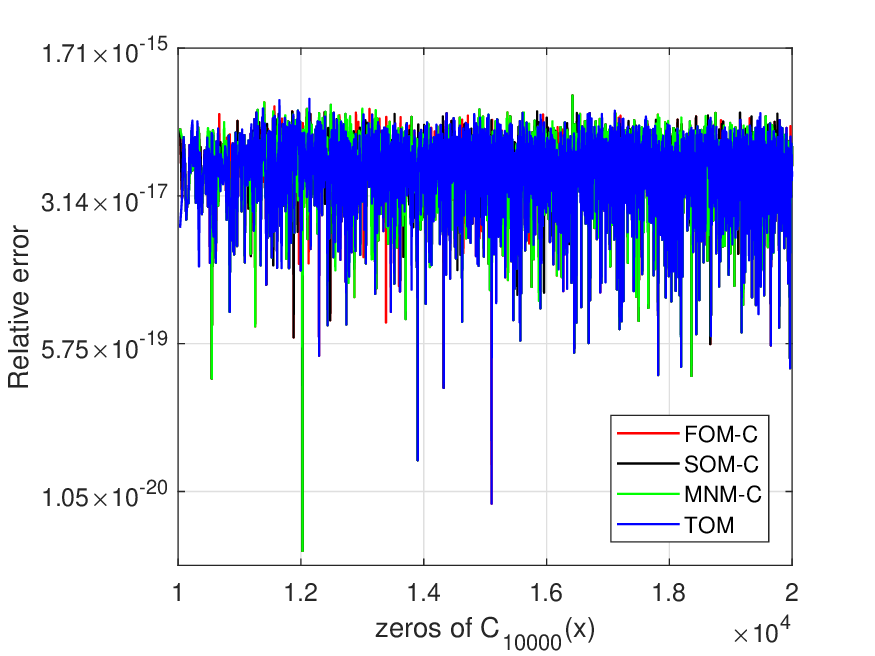}
    \caption{Relative error \eqref{error1} at zeros of $C_{10000}(x)$ in the interval $I_0$
    with $\alpha=0.75$}
    \label{re_2}
\end{subfigure}
\hfill
\begin{subfigure}[t]{0.48\textwidth}
    \centering
    \includegraphics[width=\linewidth]{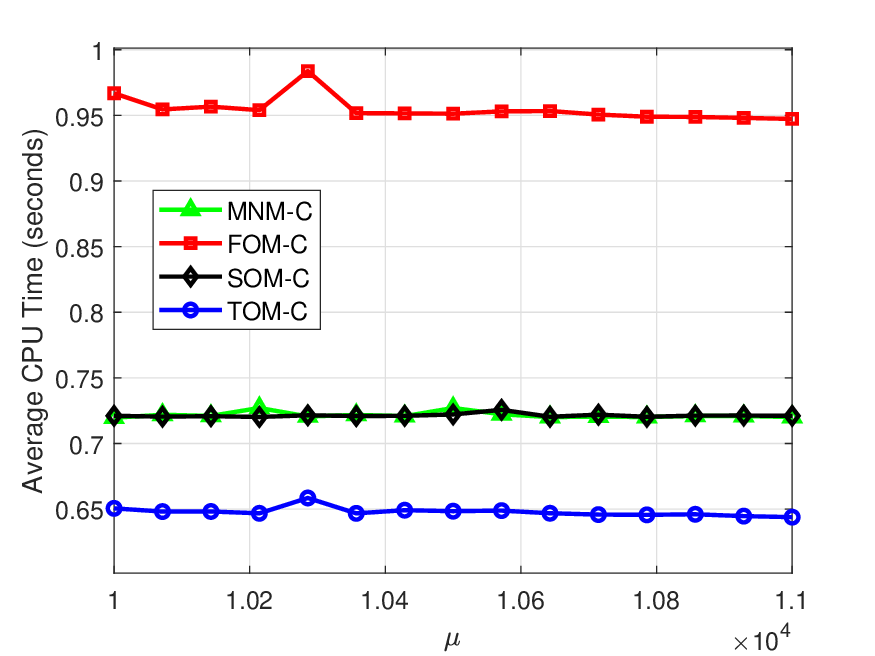}
    \caption{Average CPU runtime for computing all zeros of $c_{\mu}(x)$ in the interval $I_1$
    for various choices of $\mu \in [10000,11000]$ with $\alpha=0.75$}
    \label{CPU_cyl_a}
\end{subfigure}

\caption{FOM--C \cite{2010}, MNM--C \cite{1998}, SOM--C \cite{2003}, and TOM \eqref{iter} in finding zeros of $c_{\mu}(x)$.}
\end{figure}














\begin{table}[h]
\centering
\small
\setlength{\tabcolsep}{0.8pt}
\renewcommand{\arraystretch}{1}

\begin{tabular}{|c|cc|ccc|ccc|ccc|}
\hline
\multirow{2}{*}{\textbf{\begin{tabular}[c]{@{}c@{}}$\alpha=0.75$\\ $\mu$\end{tabular}}}
& \multicolumn{2}{c|}{\textbf{FOM-C} \cite{2010}}
& \multicolumn{3}{c|}{\textbf{MNM-C} \cite{1998}}
& \multicolumn{3}{c|}{\textbf{SOM-C} \cite{2003}}
& \multicolumn{3}{c|}{\textbf{Proposed-TOM} \eqref{iter}} \\ \cline{2-12}

& \textbf{T.Iter} & \textbf{A.Time} 
& \textbf{MR.Er-F} & \textbf{T.Iter} & \textbf{A.Time}
& \textbf{MR.Er-F} & \textbf{T.Iter} & \textbf{A.Time}
& \textbf{MR.Er-F} & \textbf{T.Iter} & \textbf{A.Time}\\ \hline

1000  & 63328 & 1.0293 & $7.41\times10^{-16}$ & 65337 & 0.6904  & $7.41\times10^{-16}$ & 65314 & 0.6882  & $7.41\times10^{-16}$ & 63725 & 0.6736  \\
3000  & 62630 & 0.9942 & $6.95\times10^{-16}$ & 66600 & 0.7021  & $6.95\times10^{-16}$ & 66580 & 0.7000  & $6.95\times10^{-16}$ & 63295 & 0.6692  \\
6000  & 61623 & 0.9675 & $7.44\times10^{-16}$ & 67645 & 0.7109  & $7.44\times10^{-16}$ & 67628 & 0.7074  & $7.44\times10^{-16}$ & 62516 & 0.6610  \\
9000  & 60664 & 0.9578 & $7.29\times10^{-16}$ & 68349 & 0.7202 & $7.29\times10^{-16}$ & 68332 & 0.7259 & $7.29\times10^{-16}$ & 61713 & 0.6581  \\
11000 & 60048 & 0.9432 & $7.91\times10^{-16}$ & 68686 & 0.7197  & $7.91\times10^{-16}$ & 68660 & 0.7206  & $7.91\times10^{-16}$ & 61189 & 0.6460  \\ \hline
\end{tabular}

\caption{Computations of all zeros of $C_{\mu}(x)$ in the interval $I_1$ for different values of $\mu$: FOM-C \cite{2010}, MNM-C \cite{1998}, SOM-C \cite{2003}, and TOM \eqref{iter}.}
\label{table_comparison_cylinder}
\end{table}






Table \ref{table_comparison_cylinder} presents the total number of iteration and the average time taken by FOM-C \cite{2010}, MNM-C \cite{seg2}, SOM-C \cite{2002}, and TOM \eqref{iter} for computing all the zeros of selected cylinder functions in the interval $I_1 = (\mu, \mu+100000)$ for various $\mu$ ranges in the interval $[1000, 11000]$ and $\alpha = 0.75$. For this problem, the iterative methods TOM and FOM-C can be written as Eq. \ref{iter_cylinder}
and 
\begin{equation*}
x_{k+1} = x_k - \frac{1}{\sqrt{\mathcal{A}(x_k)}}\arctan\left(\mathcal{A}(x_k)t_1(x_k)\right), \quad t_1(x) = \frac{C_{\mu}(x)}{C_{\mu-1}+\frac{1}{x}(\frac{1}{2}-\mu)C_{\mu}(x)}, \quad \mathcal{A}(x) = 1+\frac{\frac{1}{4}-\mu^2}{x^2},
\end{equation*}
respectively. In Table \ref{table_comparison_cylinder}, the performance of the proposed method and the second-order method\cite{seg2,2002} is compared with the fourth-order method \cite{2010}. As expected, FOM-C \cite{2010} and MNM-C \cite{seg2} take the minimum and maximum number of iterations, respectively, to find all the zeros of various $C_\mu(x)$ in $I_1$. It is interesting to note that the proposed third-order method, MNM-C \cite{seg2}, and SOM-C \cite{2002} have the same maximum relative error MR.Er-F with respect to FOM-C \cite{2010}. For all the methods the MR.Er-F is in the order of $10^{-16}$. From Table \ref{table_comparison_cylinder}, it can be seen that, although the proposed third-order method TOM requires more iterations than the fourth-order method FOM-C \cite{2010}, TOM takes less time to do the same job than FOM-C. This may be due to the computational complexity involved in repeatedly evaluating functions in the respective iterative procedures.

\section{Conclusions}
This manuscript developed a novel third-order iterative procedure to compute all the zeros of special functions that are solutions of a second-order linear ODE. This work derived sufficient conditions to ensure the global convergence of the proposed method. The well-known orthogonal polynomials, Legendre and Hermite, and the frequently used special functions, such as the Bessel and cylinder functions, satisfy these sufficient conditions for global convergence. This study develops new algorithms for finding all zeros of the Legendre and Hermite polynomials. In addition, this study develops new algorithms for finding all zeros of the Bessel and cylinder functions within a given interval. This manuscript provides extensive numerical simulations to demonstrate the theoretical results. Numerical experiments support theoretical claims. This study presents a performance comparison of the proposed method with the asymptotic-based method \cite{alex1}, fourth-order iterative methods \cite{seg1, seg11, 2010}, and second-order iterative methods \cite{2003, seg2, 2002}. In certain scenarios, numerical simulations indicate that the proposed method converges rapidly.

	\subsection*{Acknowledgement}
	Dhivya Prabhu K is thankful to the CSIR, India (Grant No. 09/1022(11054)/2021-EMR-I) for the financial support. During the preparation of this work, the authors utilised the Grammarly application to enhance the readability. After using this tool/service, the authors reviewed and edited the content as needed and take full responsibility for the content of the publication.
    
	\subsection*{Author Contributions} All figures and tables were prepared by Dhivya Prabhu K. All authors wrote and reviewed the manuscript. 
    
    \subsection*{Funding} Council of Scientific and Industrial Research, India (CSIR) provided financial support for Dhivya Prabhu K for this research. 
    
    \subsection*{Data Availability} Not applicable.
    
    \subsection*{Code availability} The source code to reproduce the results of this article is available at \cite{code}.

    \section*{Declarations}

    \subsection*{Conflict of interest} The authors declare no competing interests.

\end{document}